\documentclass[11pt, a4paper]{article}

\usepackage[ngerman,english]{babel}
\usepackage[latin1]{inputenc}
\usepackage[T1]{fontenc} 

\usepackage{amssymb}
\usepackage{amsmath}
\usepackage{amsfonts
           ,amsthm
           ,enumerate
           ,bbm 
           }

\numberwithin{equation}{section}
\usepackage[nodayofweek]{datetime}

\newcommand{\IE}{\mathbb{E}}

\newcommand{\IN}{\mathbb{N}}

\newcommand{\IP}{\mathbb{P}}
\newcommand{\IQ}{\mathbb{Q}}
\newcommand{\IR}{\mathbb{R}}

\newcommand{\cE}{\mathcal{E}}
\newcommand{\cF}{\mathcal{F}}

\newcommand{\cH}{\mathcal{H}}
\newcommand{\cI}{\mathcal{I}}

\newcommand{\cK}{\mathcal{K}}
\newcommand{\cL}{\mathcal{L}}

\newcommand{\cQ}{\mathcal{Q}}

\newcommand{\cS}{\mathcal{S}}

\newcommand{\cY}{\mathcal{Y}}
\newcommand{\cZ}{\mathcal{Z}}

\newcommand{\cHBMO}{{\cH_{BMO}}}
\newcommand{\e}{\mathrm{e}}

\newcommand{\sgn}{\text{sgn}}

\newcommand{\ud}{\mathrm{d}}
\newcommand{\uds}{\mathrm{d}s}
\newcommand{\udt}{\mathrm{d}t}

\newcommand{\udws}{\mathrm{d}W_s}
\newcommand{\udw}{\mathrm{d}W}

\newcommand{\1}{\textbf{1}}
\newcommand{\indicfunc}{1}  

\newcommand{\bit}{\begin{itemize}}
\newcommand{\eit}{\end{itemize}}

\theoremstyle{plain}
\newtheorem{theo}{Theorem}[section]
\newtheorem{lemma}[theo]{Lemma}
\newtheorem{prop}[theo]{Proposition}
\newtheorem{coro}[theo]{Corollary}

\newtheorem{defi}[theo]{Definition}

\newtheorem{remark}[theo]{Remark}

\newtheorem{assump}[theo]{Assumption}

\title{Quadratic FBSDE with generalized Burgers' type nonlinearities, perturbations and large deviations}

\author{
        \normalsize  Christoph Frei \\[8pt]
        \small  University of Alberta \\
        \small  Dept.~of Math.~and Stat.~Sciences  \\
        \small  Edmonton AB T6G 2G1\\
        \small  Canada \\
        \small  \\
        \small  cfrei@ualberta.ca\\
   \and
		\normalsize Gon\c calo Dos Reis \\[8pt]
        \small  Technische Universit\"at Berlin \\
        \small  10623 Berlin, Germany \\
        \small  and \\
				\small  CMA/FCT/UNL \\
        \small  2829-516 Caparica, Portugal \\
        \small  dosreis@math.tu-berlin.de
\vspace*{0.8cm}}

\date{ \currenttime, \ddmmyyyydate\today}

\begin{document}

\selectlanguage{english}

\maketitle

\begin{abstract}
We discuss BSDE with drivers containing nonlinearities of the type $p(y)|z|$ and $p(y)|z|^2$ with $p$ a polynomial of any degree. Sufficient conditions are given for existence and uniqueness of solutions as well as comparison results. We then connect the results to the Markovian \hbox{FBSDE} setting, discussing applications in the theory of PDE perturbation and stating a result concerning a large deviations principle for the first component of the solution to the BSDE. 
\end{abstract}

\section{Introduction}

In the past 10 years there has been an explosion of publications in the field of backward stochastic differential equations (BSDE) mainly due to their significance in optimization problems in stochastic control theory. BSDE provide a full stochastic approach to control problems which are usually found in the Hamilton-Jacobi-Bellman formalism. Moreover, in a Markovian setting they provide a connection to certain classes of parabolic PDE via the generalized Feynman-Kac formula. An application of this falls into the realm of numerical analysis, where it may be better to employ a probabilistic solver for a PDE than a standard deterministic one (because of dimensionality for instance). BSDE with driver that grow quadratically in the control variable (introduced in \cite{Kobylanski2000} and referred to as gqBSDE) are of particular importance as they appear naturally in many problems from mathematical finance, for example in the context of utility optimization with exponential or power utility functions in incomplete markets. 

In this work we study an extended class of qgBSDE where we allow for cross nonlinearities in the driver function. The canonical setting for qgBSDE, apart from a bounded terminal condition, is a driver satisfying a growth condition of the type
\[
|f(t,y,z)|\leq C(1+|y|+|z|^2)\quad \text{with } C>0.
\]
Here we extend, for the first time, the framework to a growth condition of the type
\[
|f(t,y,z)|\leq C\big( 1+|y|+(1+ |y|^k)|z| +(1+|y|^k) |z|^2\big),\ k\in\IN,\, C>0,
\]
which coupled with the corresponding modulus of continuity condition allows us to obtain existence, uniqueness and comparison results. In a second step we associate with a BSDE of this type a standard forward diffusion process (these systems are called forward-backward SDE and denoted as FBSDE) and make the link to the PDE framework via a nonlinear Feynman-Kac formula.

The results we present in this work concerning this class of BSDE serve as a stepping stone for future research, for instance, under certain conditions this type of BSDE can be seen as cash sub-additive risk measures as described in Section 8 of  \cite{ElKarouiRavanelli2009}. Moreover, they also appear in problems of optimal investment and consumption when payments of taxes take place (see \cite{EnglezosFrangosKartalaEtAl2011}).

In the PDE framework this type of nonlinearities are also of importance. These improvements on existence and uniqueness results for BSDE include, as particular cases in the PDE framework, results for the Burgers' and generalized Burgers' equations. For example, \cite{CruzeiroShamarova2009,CruzeiroShamarova2010} consider a problem in fluid mechanics that links the Navier-Stokes equation for incompressible fluid flow and FBSDE with a driver of the type $f(y,z)=yz$. But they only study that specific FBSDE.

Another example is the equation $u_t=\Delta u + u|\nabla_x u|^2$ for $x\in D$, $t\geq 0$ and where $D$ is a smooth bounded domain of $\IR^n$. This PDE relates to an FBSDE with a driver of the type $f(y,z)=y|z|^2$ that clearly falls in our setting. This type of equation belongs to a class of equations used in a wide range of applications: in geometry, as a tool to build a harmonic map that is homotopy equivalent to another given map (see e.g. \cite{Struwe1996}), in the theory of ferromagnetic materials (the Landau-Lifschitz-Gilbert equation or in models of magnetostriction) and the theory of liquid crystals (see \cite{BertschDalPassovanderHout2002} and references therein).

The third message of this paper is concerned with the theory of PDE perturbation and the same concept for FBSDE. A rough intuition is as follows: given a first-order transport PDE for which one is uncertain wether it has a solution or one is unable to ascertain its regularity, then one usually adds a vanishing regularizing term (typically a vanishing Laplacian) transforming the first order PDE into a second order parabolic PDE which has better properties (or so one hopes). Then as the regularizing term vanishes one aims at proving that part of those better properties carry to the limit. 

We discuss the PDE equivalent to a perturbation technique in the FBSDE framework. Here we look at a FBSDE system where the forward diffusion equation is perturbed by a $\sqrt{\varepsilon}\sigma \udw$ term and investigate what happens to the FBSDE solution as $\varepsilon$ vanishes. The scope of the results we are able to obtain depend on the type of perturbation we work with. With suitable assumptions we are able to go indeed far and even show a large deviations principle (LDP) for the first component of the BSDE's solution as the SDE's diffusion coefficient vanishes (LDPs for SDEs are well known, see e.g. \cite{FreidlinWentzell1998}, \cite{DemboZeitouni1993} or \cite{FengKurtz2006}). In a subsequent iteration, we give two particular applications of the theory we have just developed to emphasize the results and their limitations. The relevance of this type of results in finance is related to the work of \cite{SircarSturm2011}, where the authors look at smile asymptotics derived from a model using FBSDE as risk measures. Although the authors start with a FBSDE formulation, they quickly change to a PDE formulation in order to obtain the said asymptotics. The results we propose in this work may prove useful to show such asymtpotics without the need of such PDE results.

Concerning the techniques we use to show the large deviations principle, based on the works of \cite{DossRainero2007} and \cite{Rainero2006}, they differ quite substantially from the techniques commonly used to prove such a result for diffusion processes. In the FBSDE case one is able to directly use the contraction principle while in the SDE framework, when the diffusion coefficient depends on the state process, one is not. The reason is that the SDE's solution process when interpreted in the space of continuous functions, although measurable, needs not be continuous in that space. A rigorous derivation of an LDP for SDEs can be found for example in \cite{DemboZeitouni1993} or \cite{FengKurtz2006}. In our case we are able to take advantage of the Markovian framework in which the FBSDE is immersed and its inherent PDE characterization. This lightens considerably the complexity of the proof and, under certain conditions, allows a direct use of the contraction principle.

The work is organized as follows: in Section 2 we introduce the notation we will work with and recall for completeness some results on bounded mean oscillation (BMO) martingales. Section 3 contains the main theorem concerning existence, uniqueness, comparison and estimates for the class of BSDE we are introducing. In Section 4 we apply the results to the Markov\-ian FBSDE setting, link the existence of solution of the FBSDE to the viscosity solution of the corresponding PDE and discuss differentiability of the FBSDE with relation to the initial condition of the forward diffusion. We close with Section 5 and the results on the perturbation of FBSDE and the corresponding PDEs, a large deviations results and two applications.

\section{Preliminaries}

\subsection{Spaces and Notation}

Throughout fix $T>0$. We work on a canonical Wiener space $(\Omega, \cF,  \IP)$ carrying a
$d$-dimensional Wiener process $W = (W^1,\cdots, W^d)$ restricted to
the time interval $[0,T]$, and we denote by
$\cF=(\cF_t)_{t\in[0,T]}$ its natural filtration enlarged in the
usual way by the $\IP$-zero sets. We shall need the following
operators, and auxiliary spaces of functions and stochastic
processes: let $p\geq 2, m, n, d\in \IN$, $\IQ$ a probability
measure on $(\Omega, \cF)$. We use the symbol $\IE^\IQ$ for the
expectation with respect to $\IQ$, and omit the superscript for the
canonical measure $\IP$. For vectors $x = (x^1,\cdots, x^m)$ in
Euclidean space $\IR^m$ we write $|x| = (\sum_{i=1}^m(x^i)^2)^{\frac{1}{2}}$, we also define $\1=(1,\cdots,1)$ for the $m$-dimensional vector with all entries equal to $1$. We define by $\sgn(x):\IR\to\IR$ the sign function yielding $1$ if $x\geq 0$ and $-1$ for $x<0$. Let $B_R(x_0)$ denote the ball of radius $R\geq 0$ centered around the point $x_0$. By $\indicfunc_A$ we denote the indicator function of a set $A$.  

For a map $b:\IR^m\to\IR^d$, we denote by $\nabla b$ its Jacobian matrix whenever it exists. To denote the $j$-th first derivative of the function $b(x)$ with $x\in\IR^m$ we write $\nabla_{x_j} b$.
For a function $h(x,y):\IR^m\times\IR^d\to \IR$ we write $\nabla_x h$ or $\nabla_y h$ to refer to the first derivatives with relation to $x$ and $y$ respectively. $\Delta$ denotes the canonical Laplacian operator. We will use the operator $\partial_t$ to denote the temporal partial derivative.

We also introduce the following spaces: 
\begin{itemize}
\item $C^k_b(\IR^m,\IR^n)$ the set of $k$-times differentiable real valued maps defined on $\IR^m$ mapping onto $\IR^n$ with bounded partial derivatives up to order $k$, and $C^\infty_b(\IR^m,\IR^n)=\cap_{k\geq 1} C_b^k(\IR^m,\IR^n)$; We omit the subscript $b$ to denote the same set but without the boundedness assumptions. 

\item $L^p(\IR^m; \IQ)$ the space of $\cF_T$-measurable random variables $X:\Omega\to\IR^m$, normed by $\lVert X\lVert_{L^p}=\IE^\IQ[ \, |X|^p]^{\frac{1}{p}}$; $L^\infty$ the space of bounded random variables;

\item  $\cS^p(\IR^m)$ the space of all $\cF$-adapted processes $(Y_t)_{t\in[0,T]}$ with values in
$\IR^m$ normed by $\| Y \|_{\cS^p} = \IE[\sup_{t \in [0,T]} |Y_t|^{p}]^{\frac{1}{p}}$; $\cS^\infty(\IR^m)$ the space of bounded measurable processes;

\item $\cH^p(\IR^m, \IQ)$ the space of all $\IR^m$-valued predictable processes $(Z_t)_{t\in[0,T]}$ normed by $\|Z\|_{\cH^p} = \IE^\IQ[\big( \int_0^T |Z_s|^2 \ud s \big)^{p/2} ]^{\frac{1}{p}};$

\item $BMO(\IQ)$ or $BMO_2(\IQ)$ the space of square integrable martingales $\Phi$ with $\Phi_0=0$ and satisfying
\[\lVert \Phi \lVert_{BMO(\IQ)}= \sup_{\tau}\big\|\, \IE^\IQ\big[ \langle \Phi\rangle_T - \langle \Phi \rangle_\tau| \cF_\tau \big]\big\|_{\infty}^{1/2}< \infty,\] where the supremum is taken over all stopping times $\tau\in[0,T]$.

\item $\cH_{BMO}$ the space of $\IR^m$-valued $\cH^p$-integrable processes $(Z_t)_{t\in[0,T]}$ for all $p\geq 2$ such that $\int_0^\cdot Z_s\udws\in BMO$. We define\footnote{If $Z*W\in BMO$ then one has automatically that $Z\in \cH^p$ for all $p\geq 2$. This is a consequence of the BMO spaces, for more see the BMO results subsection below.} $\|Z\|_{\cHBMO}=\|\int Z\udw\|_{BMO}$. 
\end{itemize}
If there is no ambiguity about the underlying spaces or measures, we
also omit them as arguments in the function spaces defined above.

Constants appearing in inequalities of our proofs will for
simplicity be denoted by $C$, although they may change from line to
line.

\subsection{BMO processes and their properties}

The BMO space is an interesting space of stochastic processes satisfying
\[
L^\infty \subsetneq \ BMO \ \subsetneq \cap_{p\geq 1} L^p.
\]
For more details on BMO spaces we refer the reader to \cite{Kazamaki1994}.

In the following lemma we state some properties of BMO martingales we will frequently use.
\begin{lemma}[Properties of BMO martingales]
\label{lemma:bmoproperties}$\phantom{121}$
\begin{enumerate}[1)]
\item Given a BMO martingale $M$ with quadratic variation $\langle M\rangle$, its stochastic exponential
$\cE(M)_T := \exp\{M_T-\frac{1}{2}\langle M\rangle_T\}$
has integral $1$, and thus the measure defined by $\ud\IQ = \cE(M)_T \ud\IP$
is a probability measure.

\item For every  BMO martingale $M$, there exists $p^*>1$ such that\footnote{The number $p^*$ can be found through the function  $\Psi(x)=\big\{1+\frac{1}{x^2}\log\frac{2x-1}{2(x-1)}\big\}^{{1}/{2}}-1$ defined for all $1<x<\infty$ and verifying $\lim_{x\to1^+} \Psi(x)=\infty$ and $\lim_{x\to\infty} \Psi(x)=0$. In other words, if $\lVert M\lVert_{BMO_2}<\Psi({p^*})$, then $\cE(M) \in L^{p^*}$, see Theorem 3.1 \cite{Kazamaki1994}.} $\cE(M) \in L^{p^*}$. Moreover, there exists a constant $C_{p^*}$ depending only on $p^*$ and the BMO norm of $M$ such that for any stopping time $\tau\in[0,T]$ it holds that
\begin{align}
\label{eq:reverseholderineq:trick}
\IE[\, \cE(M)_T^{p^*} \,|\cF_t]\leq C_{p^*}\big(\cE(M)_t\big)^{p^*}
\end{align}

\item If $\|M\|_{BMO_2}<1$, then\footnote{This result is known as the John-Nirenberg inequality, see Theorem 2.2 in \cite{Kazamaki1994}.} for every stopping time $\tau\in[0,T]$ 
\begin{align}
\label{eq:johnnirenberg:ineq}
\IE\big[\,\exp\{\langle M\rangle_T-\langle M\rangle_\tau  \}\,|\cF_\tau\big]< \dfrac1{1-\|M\|^2_{BMO_2}}. 
\end{align}
In particular, if $Z*W\in BMO$, then for every $p\geq 1$ it holds that 
\begin{align}
\nonumber
&\IE\Big[\,\Big(\int_0^T |Z_s|^2\uds\Big)^p\,\Big]\leq p!\, \|Z*W\|_{BMO}^{2p}\qquad  \Rightarrow\quad Z\in\cH^{2p}.
\end{align}
Moreover, for any $p\geq 1$ and any\footnote{ 
This inequality follows from \eqref{eq:johnnirenberg:ineq}, since for every $\varepsilon\in(0,2)$ and $\delta> 0$ there exists a constant $C_{\delta,\varepsilon}$ such that  $|z|^\varepsilon\leq C_{\delta,\varepsilon}+\delta|z|^2$ for all $z$. } $\varepsilon\in(0,2)$
\begin{align}
\label{eq:BMOproperties}
& 
\IE\Big[\, \exp\Big\{\, p \int_0^T |Z_s|^{\varepsilon} \uds \,\Big\}\,\Big] \leq \widehat{C} < \infty,
\end{align}
where $\widehat{C}$ depends on $p$, $\varepsilon$ and $\|Z*W\|^2_{BMO_2}$.
\end{enumerate}
\end{lemma}

\section{Existence, uniqueness and comparison}

We now introduce the class of BSDE we will work with. The main novelty is a setting that allows for nonlinearities of the type $y^kz$ or $y^k|z|^2$ for some positive power $k\geq 1$.

\subsection{Assumptions}

Take a BSDE satisfying the following dynamics
\begin{align}
\label{general-BSDE}
Y_t &= \xi +\int_t^T f(s,Y_s,Z_s)\uds -\int_t^T Z_s \udws.
\end{align}
We refer to this stochastic equation as BSDE$(\xi,f)$. We also introduce the set of assumption under which we will be working.
\begin{assump}
\label{assump:H0}
$\xi$ is an $\cF_T$-measurable uniformly bounded random variable, i.e. for some $M>0$ we have $\| \xi \|_{L^\infty}\leq M$;
\end{assump}
\begin{assump}
\label{assump:H1}
$f:\Omega\times [0,T]\times \IR\times \IR^d\to \IR$ is an $\cF$-predictable continuous function. There exist $k\in\IN$ and positive constants $K$, $\delta$ such that for all $(\omega,t,y,z)\in \Omega \times [0,T] \times \IR \times \IR^d$ 
\begin{align}
\label{H2growth}
&|f(\omega,t,y,z)|\leq K\big( 1+|y|+(1+ |y|^k)|z| +(1+\gamma |y|^k) |z|^2\big).
\end{align}
\end{assump}
\begin{assump}
\label{assump:H2} There exist $k\in\IN$ and a positive constant $K$ such that for all $t\in[0,T]$, $y,y'\in\IR$ and $z,z'\in\IR^d$ it holds $\IP$-a.s.\,that
\begin{align}
\label{H2lip}
&|f(\cdot,t,y,z)-f(\cdot,t,y',z')| \\ 
\nonumber
&\quad \quad
\leq  K \big\{ 
\big(1+(1+\gamma|z|+\gamma |z'|)(|y|^{k-1}+|y'|^{k-1})(|z|+|z'|)\big)\big|y-y'\big|
\\
\nonumber
&
\hspace{2cm} +\big(1+|y|^k+|y'|^k+(1+\gamma|y|^k+\gamma|y'|^k)(|z|+|z'|)\big)\big|z-z'\big| \big\}.
\end{align}
\end{assump}
\begin{remark} Drivers like $f(y,z)=\nu y z$ or $f(y,z)= y|z|^2$ can be found in applications in physics (the first relates to Burger's PDE, see the introduction for remarks on the second). Drivers like $f(t,y,z)=\theta_t z+ \gamma|z|^2$ or $f(y,z)=-(ay^+-by^-)|z|^2$ with $0<a<b$ are found in applications in finance (the last one relates to cash subadditive risk measures, see e.g. \cite{ElKarouiRavanelli2009}). 
\end{remark}

\begin{remark}
In the above assumptions we write a domination in terms of a power $|y|^k$ but since $|y|^q\leq C_{p,q}+|y|^p$ for any $q\leq p$ and some constant $C_{p,q}$ it is clear that this includes any polynomial dependence of $y$ up to power~$k$.
\end{remark}

\begin{remark}
In \cite{Tevzadze2008}, the main theorem holds under his assumption (B) that imposes ``$|\nabla_y f|\leq \text{Const.}$'', a condition which does not allow for $yz$ type non-linearities. However, Proposition 1 of \cite{Tevzadze2008} partially covers our setting by allowing $y^2$ terms if an extra smallness assumption of the involved data is taken, namely that $\|\xi\|_{L^\infty}$ and $\|f(\cdot,0,0)\|_{L^\infty}$ are very small.
\end{remark}
For completeness we quote Theorem 2.3 from \cite{Kobylanski2000}. This result plays a crucial role in  proving that a BSDE under Assumption \ref{assump:H1} has a solution.  
\begin{prop}[Theorem 2.3 of \cite{Kobylanski2000}]
\label{theo:2.3Koby2000}
Let Assumption \ref{assump:H0} hold and assume a continuous $\cF$-predictable function $f:\Omega\times [0,T]\times \IR\times \IR^d\to \IR$ satisfies for any $(\omega,t)\in\Omega\times [0,T]$
\begin{align}
\label{eq:koby2000:exist:prop:condition}
|f(\omega,t,y,z)|\leq A+B|y|+F |z|^2,\quad \text{with } A,B,F\in [0,+\infty).
\end{align}
Then the BSDE($\xi$,$f$) \eqref{general-BSDE} has a solution $(Y,Z)\in \cS^\infty\times \cH^2$. The process $Y$ has continuous paths.

Moreover, there exists a unique minimal solution $(Y_*,Z_*)$  (respectively a unique maximal solution $(Y^*,Z^*)$) in the sense that the solution $(U,V)$ of BSDE$(\eta,h)$ where $f\leq h$ and $\xi\leq \eta$ (respectively $f \geq h$ and $\xi \geq \eta$) satisfies 
$Y_*\leq U$ (respectively $Y^*\geq U$).
\end{prop}
As is expected from the theory of quadratic BSDE (see \cite{Morlais2009}, \cite{ImkellerDosReis2010} or \cite{dosreis2011}), we can conclude that\footnote{Here we cannot refer to Lemma 3.1 of \cite{Morlais2009} with her Assumption H1 instead of just \eqref{eq:koby2000:exist:prop:condition}. Her extra restriction $\gamma \geq b$ restricts the methods we use in this work. Please compare her assumption H1 with our  Assumption \ref{assump:H1}.} $Z*W\in BMO$.  
\begin{lemma}
\label{lemma:ZisBMO}
Under the conditions of Proposition \ref{theo:2.3Koby2000}, we have $Z\in\cHBMO$.
\end{lemma}
\begin{proof}
We only sketch the proof since this kind of argument is known. Take $t\in[0,T]$. Assume that $\|Y\|_{\cS^\infty}$ and $\|\xi\|_{L^\infty}$ are bounded by the same constant $M$ and take a constant $\alpha$ satisfying $\alpha>2F$. Let $\tau\in [0,T]$ be a stopping time.  Applying It\^o's formula to the process $\widehat Y_t=\exp\{ \alpha Y_t \}$ between $[\tau,T]$, \eqref{eq:koby2000:exist:prop:condition} and the conditional expectation on $\cF_\tau$, we obtain:
\begin{align*}
& \widehat Y_\tau \leq \IE\Big[e^{ \alpha \xi  }
+\int_\tau^T \widehat Y_s\big[
\alpha A +\alpha B |Y_s|+(\alpha F-\frac{\alpha^2}{2})|Z_s|^2  \big]\uds\,\Big| \cF_\tau\Big]\\
\nonumber
& \qquad \Leftrightarrow 
(\frac{\alpha^2}{2}-\alpha F) e^{-\alpha M} \IE\Big[ \int_\tau^T |Z_s|^2\uds\,\Big|\cF_\tau\Big] \leq 
e^{\alpha M}\big(1 + (\alpha A +\alpha B M )T\big).
\end{align*}
Since $\alpha> 2F$ we have that $2\alpha(\alpha-2F)>0$ and hence we easily get from the definition of the BMO-norm that $Z*W\in BMO$. Moreover, from the calculations we just did, the BMO norm of $Z*W$ is bounded from above by a universal constant depending only on $\|\xi\|_{L^\infty}$, $\|Y\|_{\cS^\infty}$ and the constants $A$, $B$ and $F$.

We remark as well that as $F$ decreases the upper bound for the $\cHBMO$-norm of $Z$ also decreases and vice-versa.
\end{proof}

\subsection{The main results - Abstract BSDE setting}
We now state the main results of this section. We start with an existence and uniqueness result.
\begin{theo}[Existence]
\label{maintheo:ch3:existence}
Let Assumptions \ref{assump:H0} and \ref{assump:H1} hold. Then the BSDE \eqref{general-BSDE} has a solution $(Y,Z)$ in $\cS^\infty \times \cHBMO$. There exists an upper bound for $\|Y\|_{\cS^\infty}$ independent of $\gamma$.

Moreover, there exists a unique maximal solution $(Y,Z)$ in the sense that $\widehat Y\leq Y$ for any other possible solution $(\widehat Y,\widehat Z)$.
\end{theo}
\begin{theo}[Uniqueness]
\label{maintheo:ch3:uniqueness}
Let Assumptions \ref{assump:H0}, \ref{assump:H1} and \ref{assump:H2} hold and take $(Y,Z)$ to be the solution of BSDE \eqref{general-BSDE}. Let $\|Y\|_{\cS^\infty}\leq R$,  $\|Z\|_{\cHBMO}\leq R^{\cZ}_\gamma$ for $R,R^\cZ_\gamma\geq 0$, where $R$ is independent of $\gamma$ and $R^{\cZ}_\gamma$ is not. Define the number\footnote{The function $\Psi$ is defined in part 2) of Lemma \ref{lemma:bmoproperties}.}
$p^*=\Psi^{-1}\big(2K(1+2R^k)\sqrt{T}+4K(1+2\gamma R^k)R^\cZ_\gamma\big)$ with $q^*$ its H\"older conjugate. Assume further that $\gamma$ satisfies $(2q^*K8\gamma R^{k-1})^{\frac12}R^\cZ_\gamma<1$.

Then BSDE \eqref{general-BSDE} has a unique solution in $\cS^\infty\times \cHBMO$. 
\end{theo}
\begin{remark}
1) A small clarification is necessary concerning the existence of  $\gamma$ in Theorem \ref{maintheo:ch3:uniqueness}, since $q^*$ and $R^\cZ_\gamma$ both depend on $\gamma$ and there is the possibility of a circle argument. One argues as follows, as $\gamma$ decreases so does the upper bound  for $\|Z\|_\cHBMO$ (this is intuitive but see e.g. Lemma \ref{lemma:ZisBMO}). From the definition of $\Psi$ in Lemma \ref{lemma:bmoproperties}, the smaller the upper bound for $\|Z\|_\cHBMO$ is, the greater $p^*$ is and hence the smaller $q^*$ is. So, as $\gamma$ decreases so do $q^*$ and $R^\cZ_\gamma$, making it even easier for condition \eqref{eq:gammasmallenough} to be satisfied.

2) The proof of Theorem \ref{maintheo:ch3:uniqueness} hints that it is possible to obtain the same result but under an assumption of the type $f(Y,Z)-f(U,Z) \leq  -\alpha |Z|^2 (1+\phi(Y,U))(Y-U)$ with $\phi$ a continuous \emph{positive} function. This condition would then replace the smallness assumption of $\gamma$. We do not explore this direction.
\end{remark}
Finally, we state a comparison result:
\begin{theo}[Comparison]
\label{maintheo:ch3:comparison}
Let $i\in\{1,2\}$. Define $(Y^i,Z^i) \in \cS^\infty\times \cHBMO$ as  the solution of BSDE \eqref{general-BSDE} with terminal condition $\xi^i$ and driver $f_i$. Assume that $\xi^i$ and $f_i$ satisfy Assumptions \ref{assump:H0}, \ref{assump:H1} and \ref{assump:H2}. Assume $\gamma_i$ satisfies the condition of Theorem \ref{maintheo:ch3:uniqueness} and that\footnote{The comparison theorem still holds if instead of $f_1(t,Y^2_t,Z^2_t) \leq f_2(t,Y^2_t,Z^2_t)$ one assumes $f_1(t,Y^1_t,Z^1_t) \leq f_2(t,Y^1_t,Z^1_t)$. The proof is a straightforward modification of the proof we are giving.}
\begin{align*}
\xi^1\leq \xi^2,\ \text{ and }\ f_1(t,Y^2_t,Z^2_t) \leq f_2(t,Y^2_t,Z^2_t)\ \ \ud t \otimes \ud\IP\text{-a.s.},
\end{align*}
Then we have for all $t\in[0,T]$ that $Y^1_t\leq Y^2_t$ $\IP$-a.s. 

Moreover, if either $\xi^1<\xi^2$ or $f_1(t,Y^2_t,Z^2_t) < f_2(t,Y^2_t,Z^2_t)$ in a set of positive $\ud t\otimes \ud\IP$-measure then $Y^1_0< Y^2_0$.
\end{theo}
\begin{coro}[A priori estimate]
\label{maincoro:ch3:aprioriestimate}
Under the assumptions of Theorem \ref{maintheo:ch3:comparison} with $\gamma=0$, it holds for $\delta f_\cdot:=f_1(\cdot,Y^2_\cdot,Z^2_\cdot) - f_2(\cdot,Y^2_\cdot,Z^2_\cdot)$ that for $p\geq 2$
\begin{align}
\nonumber
\hspace{-.25cm}\|Y^1-Y^2\|_{\cS^p}^p
&\leq 
C^p_{r'} C_p C_{pr',q'}\IE\Big[\Big( |\xi^1-\xi^2| + \int_0^T |\delta f_s|\uds \Big)^{prq}\Big]^{\frac1{qr}},
\\
\label{eq:aprioriestimateofdeltaZ}
\hspace{-.25cm}\|Z^1-Z^2\|_{\cH^p}^p
&
\leq 
C \Big\{ \|\delta Y\|_{\cS^{p}}^{p} + 
\|\delta Y\|_{\cS^{2p}}^{p}
 + 
\|\delta Y\|_{\cS^{p}}^{\frac{p}2} \IE\Big[\Big(\int_0^T\hspace{-.1cm} |\delta f_s|\uds\Big)^{p}\Big]^{\frac12}
\Big\},
\end{align}
where the constants $C_{pr',q'}$, $C_p$, $C^p_{r'}$ and $C$ depend only on the constants appearing in Assumptions \ref{assump:H0}, \ref{assump:H1} and \ref{assump:H2} and the upper bounds of the BMO norms of $Z^1*W$ and $Z^2*W$. Furthermore, the numbers $r,q>1$ are related as well to the BMO norms of $Z^1*W$ and $Z^2*W$ via point 3 of Lemma \ref{lemma:bmoproperties}.
\end{coro}
The following subsections contain the proofs of the above theorems.
\subsubsection{Proof of the Theorem \ref{maintheo:ch3:existence} - Existence}

As we mentioned earlier, we now need to state and prove a sequence of results on a certain family of BSDEs that approximates the BSDE \eqref{general-BSDE}.
\subsubsection*{A ``truncation of the identity'' function family}
We start by defining a family of smooth functions that truncate the identity function, namely
\begin{defi}
For each $n\in\IN$ let ${h}_n:\IR\to\IR$ be a continuously differentiable function with the following properties:
\begin{itemize}
	\item $|h_n(x)|\leq n$, $|h_n(x)|\leq |x|$ and $|\nabla_x h_n(x)|\leq 1$ for all $x\in\IR$; 
	\item $h_n(x)=x$ for any $x \in [-(n-1), (n-1)]$ and $h_n(x)=n$ for any $x$ outside $(-n,n)$;
\end{itemize}
	The sequence $({h}_n)_{n\in\IN}$ converges locally uniformly to the identity function; the sequence $(\nabla_x h_n)_{n\in\IN}$ converges to $1$ locally uniformly.

We call the family of functions $(h_n)_{n\in\IN}$ a \emph{(differentiable) truncation of the identity}\footnote{Such a family of functions clearly exists, for an explicit example we point to Subsection 3.2.1 of \cite{dosreis2011}.}.
\end{defi}
\subsubsection*{A family of truncated drivers and their properties}
With the above defined family of functions we take a driver $f$ satisfying Assumptions \ref{assump:H1} and \ref{assump:H2} and define the sequence $(f_n)_{n\in\IN}$ through 
\[
f_n(t,y,z)=f\big(t,h_n(y),z\big),\quad \text{for }(\omega,t,y,z)\in\Omega\times[0,T]\times \IR\times \IR^d,\, n\in\IN.
\]
Using the properties of $h_n$ (along with $|z|\leq 1+|z|^2$) and in view of \eqref{H2growth}, it is clear that there exists a sequence $(C_n)_{n\in\IN}$ such that for any $(y,z)\in\IR\times\IR^d,\, n\in\IN$
\begin{align}
\label{eq:fn:growth}
&|f_n(\cdot,y,z)|\leq K\big( 1+C_n(1+|z|^2)\big)\quad \text{with } \lim_{n\to\infty} C_n=\infty,
\\
\label{eq:fn:growthindep:n}
&|f_n(\cdot,y,z)|\leq K\big( 1+|y|+(1+ |y|^k)|z| +(1+\gamma |y|^k) |z|^2\big),
\end{align}
where the latter inequality corresponds to \eqref{H2growth} for $f_n$. 

It is also easy to see that for any $(y,z),(y',z')\in\IR\times\IR^d$ and $n\in\IN$ the difference $|f_n(\cdot,y,z)-f_n(\cdot,y',z')|$ satisfies condition \eqref{H2lip} with constants independent of $n$: just use that $|h_n(x)|\leq |x|$ along with the mean value theorem on $h_n$ and the fact that $|\nabla h_n|\leq 1$.

\subsubsection*{A family of truncated BSDE and results on them}
With the family of drivers $(f_n)_{n\in\IN}$ we define a family of approximating BSDE obtained by replacing $f$ in \eqref{general-BSDE} by $f_n$ so that
\begin{align}
\label{approximative-general-BSDE}
Y^n_t &= \xi +\int_t^T f_ n(s,Y^n_s,Z^n_s)\uds -\int_t^T Z^n_s \udws.
\end{align}
The next result states the existence and uniqueness of the solution to \eqref{approximative-general-BSDE}. The argumentation here follows two steps: in the first we show existence of a solution to \eqref{approximative-general-BSDE} where we obtain upper bounds for the solution that depend on $n\in\IN$. In a second step, we twist the arguments and obtain upper bounds for the norms of $\cS^\infty$ and $\cHBMO$ independent of the truncation height $n\in\IN$.
\begin{lemma}
Let Assumptions \ref{assump:H0} and  \ref{assump:H1} hold. For each $n\in\IN$, BSDE \eqref{approximative-general-BSDE} has a (unique minimal/maximal) solution $(Y^n,Z^n)$ in $\cS^\infty\times \cHBMO$. 
\end{lemma}
\begin{proof}
This follows immediately from Proposition \ref{theo:2.3Koby2000} and Lemma \ref{lemma:ZisBMO} applied to \eqref{approximative-general-BSDE}, using that $f_n$ satisfies \eqref{eq:fn:growth}.
\end{proof}
The next lemma shows that it is possible to estimate the norms of $(Y^n,Z^n)$ solution to \eqref{approximative-general-BSDE} in $\cS^\infty \times \cHBMO$ by universal constants independent of $n\in\IN$ imposing\emph{ only} Assumptions \ref{assump:H0} and  \ref{assump:H1}. We use  \eqref{eq:fn:growthindep:n} instead of \eqref{eq:fn:growth}.
\begin{lemma}
\label{lemma:exist:uniq:ofYnZn}
Let Assumptions \ref{assump:H0} and \ref{assump:H1} hold. Let $n\in\IN$ and take $(Y^n,Z^n)$ to be the solution of BSDE \eqref{approximative-general-BSDE} as given in the previous lemma, then 
\begin{align*}
\sup_{n\in\IN}\big\{\, \|Y^n\|_{\cS^\infty} + \|Z^n\|_{\cHBMO}\,\big\}< \infty.
\end{align*}
In particular $\sup_n \|Y^n\|_{\cS^\infty}\leq e^{KT}(M+KT)$, which is independent of $\gamma$.
\end{lemma}

\begin{proof}
Let $t\in[0,T]$ and $n\in\IN$.

{\sf Step 1 - preparation:}
We start by going back to BSDE \eqref{approximative-general-BSDE}. Let $(Y^n,Z^n)$ solve BSDE \eqref{approximative-general-BSDE}, then we can decompose the growth of driver $f_n(t,Y_t^n,Z^n_t)$ by using the growth assumption \eqref{eq:fn:growthindep:n} in the following way: 
\begin{align*}
|f_n(t,Y_t^n,Z^n_t)|&\leq K + \beta^n_t Y^n_t+ b^n_t \cdot Z^n_t,
\end{align*}
where $\beta^n$ and $b^n$ are two processes valued in $\IR$ and $\IR^d$ respectively defined by: (the number $k$ follows from Assumption \ref{assump:H1})
\begin{align*}
\beta^n_t:=K\,\sgn(Y^n_t),\ \
\text{and}\ \ b^n_t:= K\,\frac{1+|Y^n_t|^k+(1+\gamma |Y^n_t|^k)|Z^n_t|\,}{|Z_t^n|}Z_t^n\1_{\{|Z_t^n|\neq 0 \}} .
\end{align*}
It is clear that $\sup_{n\in \IN,\,t\in[0,T]}\|\beta^n\|_{\cS^\infty}<\infty$ and $(\beta^n Y^n)= K|Y^n|$. 
Moreover, $b^n\cdot Z^n=\big(1+|Y^n|^k+(1+\gamma |Y^n|^k)|Z^n|\big)|Z^n|$ and since $Y^n\in\cS^\infty$, $Z^n\in\cHBMO$ and \eqref{H2growth} we have
\[
|b^n|\leq K\big\{1+|Y^n|^k+\big(1+\gamma |Y^n|^k\big) |Z^n|\big\}\quad \Rightarrow \quad b^n*W\in BMO.
\]
Hence the probability measure $\IQ^n$ with Radon-Nikodym density $\ud \IQ^n/\ud \IP = \cE\big(b^n*W\big)_T$ is well defined and with relation to which $W^{\IQ^n}_\cdot=W_\cdot-\int_0^\cdot b^n_s\uds$ is a Brownian motion.

{\sf Step 2 - uniform $\cS^\infty$ bound for $Y^n$:} We set $e_t^n:=\exp\{ \int_0^t \beta^n_s\uds\}$, which satisfies $\e^{-KT}\leq e^n\leq \e^{KT}$, and use a change of measure to show that an upper bound for $\|Y^n\|_ {\cS^\infty}$ can be obtained independently of $n$ (and $\gamma$).

We take $(Y^n,Z^n)$ as the solution of BSDE \eqref{approximative-general-BSDE} and with the help of processes $\beta^n$, $b^n$ and $e^n$ defined above we write the BSDE for $e_t^n Y^n_t$ under $\IQ^n$ via It\^o's formula. We have then
\begin{align}
\nonumber
&e_t^n Y^n_t 
\leq  e^n_T \xi - \int_t^T e^n_s Z^n_s \udws
+ \int_t^T e_s^n[ -\beta^n_s Y^n_s+|f_n(s,Y^n_s,Z^n_s)|\, ] \uds 
\\ \nonumber
&\phantom{e_t^n Y^n_t} 
\leq e^n_T \xi - \int_t^T e^n_s Z^n_s \udws
+ \int_t^T e_s^n[ -\beta^n_s Y^n_s+ K+\beta^n_s Y^n_s + b^n_s\cdot Z^n_s\, ] \uds 
\\ \label{eq:lin:and:measurechange:argument:for:BSDE}
&\phantom{e_t^n Y^n_t} 
\leq  \e^{KT}M - \int_t^T e^n_s Z^n_s\udws^{\IQ^n}+ \e^{KT}K(T-t).
\end{align}
where we used the properties of the process $e^n$, the measure change to $\IQ^n$ and that $\xi\in{L^\infty}$. We obtain by taking $\IQ^n$-conditional expectation that for some positive constant $C$ and for all $t\in[0,T]$
\begin{align*}
&Y^n_t \leq e^{KT}(M+KT) \ \Rightarrow \ \sup_{n\in\IN}\| Y^n \|_{\cS^\infty}\leq \e^{KT}(M+KT).
\end{align*}
where the last conclusion follows from using similar arguments to obtain a lower bound for $Y^n$. 

{\sf Step 3 - uniform $\cHBMO$ bound for $Z^n$:} The argument uses that we already know that $\sup_{n\in\IN} \|Y^n\|_{\cS^\infty}<\infty$. This yields the existence of  $n^*\in\IN$ such that for any $n\in\IN$ with $n\geq n^*$
\begin{align}
\label{eq:truncdriver:for:big:n}
f_n(t,Y^n,Z^n)
&
=f_{n^*}(t,Y^n,Z^n)=f(t,Y^n,Z^n).
\end{align}
The arguments used to prove Lemma \ref{lemma:ZisBMO} imply that the upper bound for the BMO norm of $Z*W$ depends only on the problem data and an upper bound for $\|Y\|_{\cS^\infty}$. Hence we get here that $\sup_{n\in \IN} \|Z^n\|_{\cHBMO} < \infty$.
\end{proof}
\subsubsection*{The proof of the existence theorem}
We are now ready to prove the existence theorem. Lemma \ref{lemma:exist:uniq:ofYnZn} yields a sequence $(Y^n,Z^n)_{n\in\IN}$ of processes belonging to $\cS^\infty\times \cHBMO$ that solve BSDE \eqref{approximative-general-BSDE} for each $n\in\IN$. Moreover, the processes' respective norms are bounded uniformly in $n$.  We need only to conclude that a limit for the sequence $(Y^n,Z^n)_{n\in\IN}$ exists and that the said limit solves BSDE \eqref{general-BSDE} under Assumption \ref{assump:H1}.
\begin{proof}[Proof of Theorem \ref{maintheo:ch3:existence}]
 In view of \eqref{eq:truncdriver:for:big:n} it is clear that for each $n>n^*$ pairs $(Y^n,Z^n)$ solve both the BSDE \eqref{approximative-general-BSDE} and \eqref{general-BSDE}. And hence we obtain the ``existence of solution'' result.

Moreover, since for each $n>n^*$ there exists a maximal solution to \eqref{approximative-general-BSDE}, we have the existence of a (infinite amount of) maximal solutions to \eqref{general-BSDE}. It is clear that the maximal solution is unique, in the sense given in the theorem's statement. If $(Y,Z)$ and $(\widehat Y, \widehat Z)$ are two maximal solutions then we have that $\widehat Y \leq Y$ and also $Y\leq \widehat Y$. This implies also the uniqueness of a maximal solution of BSDE \eqref{general-BSDE}.
\end{proof}

\subsubsection{Proof of Theorem \ref{maintheo:ch3:uniqueness} - Uniqueness}
We now present the proof of the uniqueness result. As usual to prove existence, one needs only bounds on the growth of the involved functions but to prove uniqueness one needs to control the modulus of continuity of the involved functions, hence for uniqueness one additionally needs Assumption~\ref{assump:H2}.

Unfortunately, the arguments we present here do not allow for general choices of $\gamma$ (from Assumption \ref{assump:H1} or \ref{assump:H2}). Here $\gamma$ has to be small enough. This smallness enters in play in \eqref{eq:uniq:eiswelldefined} below. In the later Subsection \ref{subsec:particularcase} we discuss a particular situation where it is possible to solve the BSDE for any value of $\gamma$.

\begin{proof}[Proof of Theorem \ref{maintheo:ch3:uniqueness}]
Let $t\in[0,T]$ and assume that $(Y,Z)$ and $(U,V)$, belonging to the space $\cS^\infty\times \cHBMO$, are two solutions to the BSDE \eqref{general-BSDE}. The $\cS^\infty$-norms of $Y,U$ are universally bounded by a constant, say $R$, independently of  $\gamma$ (see Lemma \ref{lemma:exist:uniq:ofYnZn}); the $\cHBMO$-norm of $Z$, $U$ are bounded by a constant $R^\cZ_\gamma$ that depends on $\gamma$ and decreases as $\gamma$ decreases (see the proof of Lemma \ref{lemma:ZisBMO}).

Now, define $\delta Y=Y-U$, $\delta Z=Z-V$ and the processes 
\begin{align}
\label{auxeq:uniq:exp}
\Gamma_t&:=\frac{f(t,Y_t,Z_t)-f(t,U_t,Z_t)}{Y_t-U_t}\indicfunc_{\{Y_t-U_t\neq 0\}},\quad 
\hat{e}_t:=\exp\Big\{ \int_0^t  \Gamma_s \uds \Big\},\\
\label{auxeq:uniq:rdnydensity}
&\qquad \text{and}\quad \hat{b}_t=\frac{f(t,U_t,Z_t)-f(t,U_t,V_t)}{|Z_t-V_t|^2}(Z_t-V_t)\indicfunc_{\{|Z_t-V_t|\neq 0\}}.
\end{align}
Concerning the above defined processes, we remark that for $\hat{b}$ we have from \eqref{H2lip} (and the triangular inequality of the $\cHBMO$-norm) that 
\begin{align*}
|\hat{b}|&\leq K\big\{ 1+2R^k+(1+2\gamma R^k)(|Z|+|V|)\big\},\\
\|\hat{b}\|_{\cHBMO}&\leq \widehat{C}:= 2K\big\{(1+2R^k)\sqrt{T}+2(1+2\gamma R^k)R^\cZ_\gamma\big\}.
\end{align*}
Since $Z,V\in\cHBMO$, it is easy to see that the process $\hat{b}*W\in BMO$ and hence the probability measure $\widehat{\IQ}$ with Radon-Nykodim density $\ud\widehat{\IQ}/\ud\IP = \cE(\hat{b}*W)$ is well defined and $W^{\widehat{\IQ}}_\cdot = W_\cdot - \int_0^\cdot \hat{b}_s\uds$ is a $\widehat{\IQ}$-Brownian motion.

Moreover, from Lemma \ref{lemma:bmoproperties} and defining $p^*=\Psi^{-1}(\widehat{C})$ with $q^*$ its H\"older conjugate we have that $\cE(\hat{b}*W)\in L^{p^*}$.

In view of \eqref{H2lip} we have that $|\Gamma| \leq K(1+4R^{k-1}|Z|+8\gamma R^{k-1}|Z|^2)$. Since $Z\in\cHBMO$, the properties of BMO martingales, namely \eqref{eq:BMOproperties}, combined with the behavior of $R^\cZ_\gamma$ (described in the proof of Lemma \ref{lemma:ZisBMO}) imply that it is possible to find a $\gamma$ small enough such that\footnote{In short $|\Gamma|\leq \gamma |Z|^2$ and we are faced with the question of the integrability of $\hat{e}$. From the John-Nirenberg inequality in \eqref{eq:johnnirenberg:ineq} we are only able to ensure the integrability of $\hat{e}$ when the constant $\gamma$ is sufficiently small. In general the quadratic variation of a BMO process $Z*W$ is not exponentially integrable.}: 
\begin{align}
\label{eq:gammasmallenough}
(2q^*K8\gamma R^{k-1})^{\frac12}< \dfrac{1}{R^\cZ_\gamma}\leq\dfrac{1}{\|Z*W\|_{BMO}}.
\end{align}
In view of \eqref{eq:johnnirenberg:ineq} and the strictness of \eqref{eq:gammasmallenough} it is always possible to find a small enough $\varepsilon>0$ such that
\begin{align}
\label{eq:uniq:eiswelldefined}
\hat{e}_T^{2q^*}\in L^{1+\varepsilon}.
\end{align}
It is also clear that $\hat{e}$ has continuous paths.

{\sf Step 1 - Uniqueness of the solution's first component:} With the help of the processes $\hat{b}$ and $\hat{e}$ defined above we proceed as in step 2 of the proof of Theorem \ref{maintheo:ch3:existence} (we skip some details and point the reader to \eqref{eq:lin:and:measurechange:argument:for:BSDE}). Using It\^o's formula, we write a BSDE for $(\hat{e}_t \delta Y_t)$ under the $\widehat{\IQ}$-measure:
\begin{align}
\nonumber
& \hat{e}_t \delta Y_t +\int_t^T \hat{e}_s \delta Z_s\udws^{\widehat{\IQ}} \\
&\qquad = 0 
	+ \int_t^T \hat{e}_s[\,
-\Gamma_s \delta Y_s + f(s,Y_s,Z_s)-f(s,U_s,Z_s)] \uds =0
\label{comparison-Uniq-BSDE:2nd:line}
\end{align}
We obtain $\hat{e}_t\, \delta Y_t=0$ $\widehat{\IQ}$- and $\IP$-a.s.\! for all $t\in[0,T]$ once we justify that the stochastic integral $\int_0^\cdot \hat{e}_s \delta Z_s\udws^{\widehat{\IQ}}$ is indeed a true $\widehat{\IQ}$-martingale. 
For this we have to essentially prove the square integrability of the process, i.e. the finiteness of
\begin{align*}
\IE^{\widehat{\IQ}}\big[\int_0^T |\hat{e}_s|^2 |\delta Z_s|^2\ud s\big]
&
\leq 
\IE\Big[ \cE(\hat{b}*W)\, |\hat{e}_T|^2\,\int_0^T |\delta Z_s|^2\ud s\Big]
\\
&
\leq 
\IE\Big[ \Big(\cE(\hat{b}*W)\Big)^{p^*}\Big]^{\frac1{p^*}}
\IE\Big[ |\hat{e}_T|^{2q^*}\Big(\int_0^T |\delta Z_s|^2\ud s\Big)^{q^*}\Big]^{\frac{1}{q^*}}.
\end{align*}
Combining now with the fact that $\cE(\hat{b}*W)\in L^{p^*}$, that $\delta Z_s \in \cHBMO$ and hence $\delta Z\in \cH^p$ for all $p\geq 2$ and finally \eqref{eq:uniq:eiswelldefined}, we can obtain the sought conclusion.

At this point it is clear that $\hat{e}_t\delta Y_t=0$ for every $t\in[0,T]$ $\IP$-a.s. We now want to conclude that we also have $\delta Y=0$. The continuity of $\hat{e}$ (and $\delta Y$) yield that there exists a set $A\in\Omega$ satisfying $\IP[A]=0$ such that
\[
\text{for all } t\in[0,T], \  \omega \in \Omega \backslash A \ \text{ we have } \hat{e}_t(\omega)>0.
\]
Now, given the positivity of $(\hat{e}_t)_{t\in[0,T]}$ we can conclude that $Y_t=U_t$ for any $t\in[0,T]$ $\IP$-a.s.

{\sf Step 2 -  Uniqueness of the solution's second component:} We are missing only the uniqueness proof for the control component of the BSDE. Here we return to \eqref{comparison-Uniq-BSDE:2nd:line} and take advantage of the fact that we already know that $\delta Y=0$. We have, by using It\^o's Isometry, that $\IE^{\widehat{\IQ}}[\int_0^T |\hat{e}_s|^2|\delta Z_s|^2\uds]=0$, from here using the positivity of $(\hat{e}_t)_{t\in[0,T]}$ (see the above step), we conclude that $|\delta Z_t|=|Z_t-V_t|=0$ $\udt\otimes\IP$-a.s. 
This concludes the proof.
\end{proof}

\subsubsection{Proof of Theorem \ref{maintheo:ch3:comparison} and Corollary \ref{maincoro:ch3:aprioriestimate} - Comparison}
We now prove Theorem \ref{maintheo:ch3:comparison} and Corollary \ref{maincoro:ch3:aprioriestimate}. The first parts of these proofs are very similar to that of Theorem \ref{maintheo:ch3:uniqueness} and so we do not give the full details.
\begin{proof}[Proof of Theorem \ref{maintheo:ch3:comparison}] Let $t\in[0,T]$.  Define $\delta Y=Y^1-Y^2$, $\delta Z=Z^1-Z^2$, $\delta \xi = \xi^1-\xi^2$ and $\delta f_\cdot = f_1(\cdot,Y^2_\cdot,Z^2_\cdot) - f_2(\cdot,Y^2_\cdot,Z^2_\cdot)$. 

We define as well, processes similar to those of \eqref{auxeq:uniq:exp} and 
\eqref{auxeq:uniq:rdnydensity}, namely:
\begin{align}
\label{auxeq:comp:exp}
\tilde{\Gamma}_t&:=\frac{f_1(t,Y^1_t,Z^1_t)-f_1(t,Y^2_t,Z^1_t)}{Y^1_t-Y^2_t}\indicfunc_{\{Y^1_t-Y^2_t\neq 0\}},\quad 
\tilde{e}_t:=\exp\Big\{ \int_0^t  \tilde{\Gamma}_s \uds \Big\},\\
\label{auxeq:comp:rdnydensity}
&\quad \text{and}\quad \tilde{b}_t=\frac{f_1(t,Y^2_t,Z^1_t)-f_1(t,Y^2_t,Z^2_t)}{|Z^1_t-Z^2_t|^2}(Z^1_t-Z^2_t)\indicfunc_{\{|Z^1_t-Z^2_t|\neq 0\}}.
\end{align}
And as before, with $\tilde{b}*W$ we define a new probability measure $\widetilde{\IQ}$ with Radon-Nikodym density $\ud\widetilde{\IQ}/\ud\IP=\cE(\tilde{b}*W)$ and $W^{\widetilde{\IQ}}_\cdot=W_\cdot-\int_0^\cdot \tilde{b}_s\uds$ is a $\widetilde{\IQ}$-Brownian motion.

From the BSDE for $\delta Y$, and It\^o's formula we can write a BSDE for the process $\tilde{e}_t \delta Y_t$  as
\begin{align}
\nonumber
\tilde{e}_t\delta Y_t&= \tilde{e}_T\delta \xi -\int_t^T \tilde{e}_s \delta Z_s\udws^{\widetilde{\IQ}} + \int_t^T\tilde{e}_s  \, \delta f_s \uds
\\
& \label{eq:aprioriestimate:proofofmmaincomptheo}
= \IE^{\widetilde{\IQ}}\Big[ \tilde{e}_T\delta \xi + \int_t^T\tilde{e}_s  \, \delta f_s \uds  \,\Big|\cF_ t\Big].
\end{align}
Using that $\tilde{e}_t>0$ and the theorem's assumptions, $\delta \xi \leq 0$ and $\delta f_t\leq 0$, we conclude that $\tilde{e}_t\delta Y_t\leq 0$ and hence that for all $t\in[0,T]$ it holds that $\delta Y_t=Y^1_t-Y^2_t\leq 0$ $\widehat{\IQ}$-a.s. and hence also $\IP$-a.s.

We see also that at $t=0$ if $\delta \xi <0$ or if $\delta f<0$ in a set of positive $\ud t \otimes \ud \IP$-measure, then the inequality is strict and we conclude that $\delta Y_0=Y^1_0-Y^2_0<0$, which proves the result.
\end{proof}
We finish this subsection with the proof of Corollary \ref{maincoro:ch3:aprioriestimate}. This proof is close to that of Lemma 3.2 in \cite{ImkellerDosReis2010} but with a different argumentation in what the linearization trick is concerned.
\begin{proof}[Proof of Corollary \ref{maincoro:ch3:aprioriestimate}]$\phantom{1}$

{\sf Step 1 -  The estimate for $Y^1-Y^2$:} To prove this corollary we stop at the point of inequality \eqref{eq:aprioriestimate:proofofmmaincomptheo} in the proof of Theorem \ref{maintheo:ch3:comparison} and continue it from there but in a different fashion. Recall that $\tilde{e}$ and $\tilde{\Gamma}$ were defined in \eqref{auxeq:comp:exp} and \eqref{auxeq:comp:rdnydensity}, and from the corollary's statement, it is assumed that $\gamma=0$, this means that \eqref{eq:uniq:eiswelldefined} now holds without any restriction on $p$.

We remark as well that for any $0\leq t\leq s\leq T$
\begin{align*}
&\tilde{e}_s \tilde{e}_t^{-1}=\exp\Big\{\int_t^s \tilde\Gamma_r \ud r\Big\}\leq A_T
:=\exp\Big\{\int_0^T K(1+4R^{k-1}|Z^1_s|) \uds\Big\}.
\end{align*}
Since $Z^1\in\cHBMO$, the random variable $A_T$ is integrable in view of \eqref{eq:BMOproperties}. To shorten the notation, we define $\cE_T:=\cE(\tilde{b}*W)_T$ as the density of the probability measure $\widetilde{\IQ}$. We have then
\begin{align*}
\delta Y_t
&
\leq \IE^{\widetilde{\IQ}}\Big[ \tilde{e}_T (\tilde{e}_t)^{-1} |\delta \xi| + (\tilde{e}_t)^{-1} \int_t^T\tilde{e}_s  \, |\delta f_s| \uds  \,\Big|\cF_ t\Big]\\
&
\leq \IE^{\widetilde{\IQ}}\Big[ A_T\Big(|\delta \xi| + \int_0^T  \, |\delta f_s| \uds  \Big)\,\Big|\cF_ t\Big].
\end{align*}
From the above inequality, it follows that 
\begin{align*}
\delta Y_t
&\leq \cE_t^{-1}\IE^{\IP}\Big[\, \cE_T\, A_T\, \Big(|\delta \xi| + \int_0^T  \, |\delta f_s| \uds  \Big)\,\Big|\cF_ t\Big]
\\
&
\leq \cE_t^{-1}\IE^{\IP}\big[\, \cE_T^{r'}\big]^{\frac1{r'}}\,\IE^{\IP}\Big[A_T^r\, \Big(|\delta \xi| + \int_0^T  \, |\delta f_s| \uds  \Big)^r\,\Big|\cF_ t \Big]^{\frac1r}
\\
&
\leq C_{r'} \IE^{\IP}\Big[A_T^r\, \Big(|\delta \xi| + \int_0^T  \, |\delta f_s| \uds  \Big)^r\,\Big|\cF_ t \Big]^{\frac1r},
\end{align*}
where the last lines follow from $\tilde{b}*W\in BMO$ combined with \eqref{eq:reverseholderineq:trick} and H\"older's inequality (with $1/r+1/{r'}=1$) for $r'>1$ with $\cE_T\in L^{r'}(\IP)$.

Let now $p\geq 1$ and apply Doob's inequality (and concave Jensen's inequality) to conclude that
\begin{align*}
\IE[\sup_{t\in[0,T]}|\delta Y_t|^p] 
&
\leq C^p_{r'} C_p \IE\Big[A_T^{pr} \Big( |\delta\xi| + \int_0^T |\delta f_s|\uds \Big)^{pr}\Big]^{\frac1r}
\\
&
\leq C^p_{r'} C_p C_{pr',q'}\IE\Big[\Big( |\delta\xi| + \int_0^T |\delta f_s|\uds \Big)^{prq}\Big]^{\frac1{qr}},
\end{align*}
where we applied once again H\"older's inequality with exponents $1/q+1/q'=1$ and version \eqref{eq:BMOproperties} of \eqref{eq:johnnirenberg:ineq} for $q'>1$ with $A_T^{pr}\in L^{q'}(\IP)$.

{\sf Step 2 -  The estimate for $Z^1-Z^2$:} We follow the notation introduced above. The estimate here is obtained using standard BSDE techniques, so we shorten the proof a little. We first remark that 
\begin{align*}
|f_1(\cdot,Y^1_\cdot,Z^1_\cdot)-f_1(\cdot,Y^2_\cdot,Z^2_\cdot)|\leq L^y_\cdot |\delta Y_\cdot|+L^z_\cdot|\delta Z_\cdot|,
\end{align*}
where, using \eqref{H2lip} with $\gamma=0$ and that $\|Y^i\|_\cS^\infty<R$ for $i\in\{1,2\}$,
\begin{align*}
L^y_\cdot :=K(1+2R^{k-1}(|Z^1_\cdot|+|Z^2_\cdot|)
\ \text{ and } \ 
L^z_\cdot :=K(1+2R^{k}+|Z^1_\cdot|+|Z^2_\cdot|).
\end{align*}
Applying It\^o's formula applied to $|\delta Y_t|^2$ over the interval $[0,T]$, using the ``Lipschitz'' assumptions of $f$, Young's inequality (with parameter $2$) and observe that with $\delta f$ defined as $\delta f_\cdot=f_1(\cdot,Y^2_\cdot,Z^2_\cdot) - f_2(\cdot,Y^2_\cdot,Z^2_\cdot)$, we obtain 
\begin{align*}
& |\delta Y_0|^2 +\int_0^T 2 \delta Y_s \delta Z_s\udws
\\
&\qquad 
= 
 |\delta \xi|^2 + \int_0^T \big[\, 2\delta Y_s \big(f_1(s,Y^1_s,Z^1_s) - f_2(s,Y^2_s,Z^2_s)\big)
 - |\delta Z_s|^2 \big]\uds
\\
&\qquad  
 \leq 
 |\delta \xi|^2 + \int_0^T \big[\, 2(L^y_s+(L^z_s)^2)|\delta Y_s|^2 + 2|\delta Y_s|\,|\delta f_s|  - \frac12|\delta Z_s|^2 \big]\uds
\\
& \Rightarrow  
\frac12 \| \delta Z \|_{\cH^{2p}}^{2p} \leq C \Big\{ \|\delta Y\|_{\cS^{2p}}^{2p} + 
\|\delta Y\|_{\cS^{4p}}^{2p}\IE\Big[\Big(\int_0^T [L^y_s +(L^z_s)^2]\uds\Big)^{2p}\Big]^{\frac12}
\\
& \hspace{1.5cm} 
\quad + 
\|\delta Y\|_{\cS^{2p}}^{p} \IE\Big[\Big(\int_0^T |\delta f_s|\uds\Big)^{2p}\Big]^{\frac12}
 + \varepsilon \|\delta Y\|_{\cS^{4p}}^{2p} +\frac1\varepsilon  \|\delta Z \|_{\cH^{2p}}^{2p}
\Big\},
\end{align*}
where the last line follows from: reordering the terms, taking absolute values, exponentiating to power $p$ and taking expectation, then use It\^o's Isometry, the fact that $\delta Y\in \cS^\infty$ and Young's inequality for some $\varepsilon>0$. Since $Z^i\in \cHBMO$ for $i\in\{1,2\}$ the processes $L^y, L^z\in\cHBMO$, now with the appropriate choice of $\varepsilon$, estimate \eqref{eq:aprioriestimateofdeltaZ} follows.
\end{proof}

\subsection{A remark on Theorems \ref{maintheo:ch3:existence} and \ref{maintheo:ch3:uniqueness}}
\label{subsec:particularcase}

Although we managed to show that a solution to BSDE \eqref{general-BSDE} exists under Assumption \ref{assump:H1} for any $\gamma\geq 0$, we were only able to show a comparison result when $\gamma$ was small enough (see \eqref{eq:uniq:eiswelldefined}). This is due to a limitation of the mathematical tools available, namely that a small $\gamma$ ensured that the $\Gamma$'s in \eqref{auxeq:uniq:exp} and \eqref{auxeq:comp:exp} had the required integrability properties. 

But there are situations where one is able to show uniqueness for any value of $\gamma$. We next discuss about the possibility of linearizing a driver, e.g. $f(y,z)=g(y) |z|^2$ with a convenient $g$ function, through an invertible transformation. The idea is similar to the usual quadratic driver case, $f(z)=|z|^2/2$ and the transformation $P_t=\exp Y_t$.

\subsubsection*{A linearization transformation}
Linearization for standard quadratic BSDE is well understood:  through the transformation $P_t=\exp\{\gamma Y_t\}$ and $Q_t=Z_tP_t$, the BSDE
\[
Y_t=\xi+\int_t^T \frac{\gamma}2|Z_s|^2\uds-\int_t^T Z_s\udws
\]
becomes
\[
P_t=e^{\gamma \xi}-\int_t^T Q_s\udws \ \ \Rightarrow \ \ 
Y_t=\frac1\gamma\log \IE[ e^{\gamma \xi} |\cF_t].
\]
It is possible to use the same type of arguments and find linearizing transformations for BSDE with drivers of the form $f(y,z)=g(y)|z|^2$ for an integrable function $g$, by solving the ODE\footnote{Apply It\^o's formula to $\Phi(Y_t)$ and organize the terms inside the Lebesgue integral in a convenient way. Note that we already know that $Y\in\cS^\infty$ and can easily verify that $\int \Phi'(Y)Z\ud W$ is a true martingale.}
\[
g(y)\Phi'(y)-\frac12\Phi''(y)=0,\quad \Phi(0)=0,\ \Phi'(0)=1.
\]
The above ODE is solved by
\[
\Phi'(y)=\exp\Big\{\int_0^y 2g(x)\ud x \Big\}
\quad\text{and}\quad
\Phi(y)= \int_0^y \Phi'(x)\ud x.
\]
Since $g$ is integrable by assumption, we have $\Phi'(y)\neq 0$ for any $y\in\IR$ so that $\Phi$ is invertible and hence the unique solution of the transformed BSDE yields the unique solution of the original BSDE.

We also point out that this type of transformation works well in the one dimensional setting. In the multidimensional case one quickly finds problems, see the Remark \ref{remark:multidimcase} below.

\section{FBSDEs and PDEs}
\label{section:FBSDEandPDE}
In this section we discuss the framework of Markovian FBSDE and their connection to PDEs. The results here open way to the later results on PDE perturbation and large deviations. 

Here we particularize the previous results to the framework of \emph{decoupled} Forward-Backward SDE (FBSDE), where we assume that the randomness of the terminal condition and driver have origin in a diffusion process. Namely, we define an FBSDE, with solution $(X,Y,Z)$, as the following system of stochastic equations: let $m\in\IN$ and take $(t,x)\in [0,T] \times \IR^m$ and for $s\in[t,T]$
\begin{align}
\label{FBSDE:SDE}
X^{t,x}_s&= x + \int_t^s b(r,X^{t,x}_r)\ud r+\int_t^s \sigma(r,X^{t,x}_r)\udw_r,
\\
\label{FBSDE:BSDE}
Y^{t,x}_s&= g(X_T^{t,x}) + \int_s^T f(r,X_r^{t,x},Y_r^{t,x},Z_r^{t,x})\ud r-\int_s^T Z_r^{t,x}\udw_r.
\end{align}
The involved functions $b$, $\sigma$, $g$ and $f$ satisfy:
\begin{assump}
\label{assump:H0:SDE}
$b:[0,T]\times\IR^m\to \IR^m$ and $\sigma:[0,T]\times \IR^m\to \IR^{m\times d}$ are continuous functions for which there exists a constant $K>0$ such that  $\sup_{t\in[0,T]}\{\,|b(t,0)|+|\sigma(t,0)|\,\}\leq K$. Moreover $b$ and $\sigma$ satisfy standard Lipschitz conditions in the spatial variables.
\end{assump}
\begin{assump}
\label{assump:H1:SDEBSDE}
$g:\IR^m\to\IR$ and $f:[0,T]\times \IR^m \times \IR\times \IR^d\to \IR$ are continuous functions. $g$ satisfies a standard Lipschitz condition and is uniformly bounded by a constant $M$. Let $k\in\IN$ and a positive constant $K$ such that for all $(t,x,y,z)\in[0,T] \times \IR^m \times \IR \times \IR^d$ 
\begin{align}
\label{assump:FBSDE:growth}
&|f(t,x,y,z)|\leq K\big( 1+|y|+(1+ |y|^k)|z| +|z|^2\big).
\end{align}
Let $t\in[0,T]$, $x,x'\in\IR^m$, $y,y'\in\IR$ and $z,z'\in\IR^d$ it holds that
\begin{align}
\label{assump:FBSDE:lip}
&|f(t,x,y,z)-f(t,x',y',z')|
\\
\nonumber
&\qquad
\leq  K \big\{ 
\big(1+|y|+|y'|+|y|^k|z|+|y'|^k|z'|+|z|^2+|z'|^2\big)\big|x-x'\big| 
\\ 
\nonumber
&\hspace{2.4cm}
+\big(1+(|y|^{k-1}+|y'|^{k-1})(|z|+|z'|)\big)\big|y-y'\big|
\\
\nonumber
&
\hspace{3.2cm} +\big(1+|y|^k+|y'|^k+(|z|+|z'|)\big)\big|z-z'\big| \big\}.
\end{align}
\end{assump}

\subsection{Particularizing the main results - Decoupled FBSDE}

We now state a particularization of the main theorems in the previous section to the framework of FBSDE.
\begin{prop}
\label{theo:FBSDEexistenceanduniqueness}
Take $(t,x)\in[0,T]\times \IR^m$ and let Assumptions \ref{assump:H0:SDE} and \ref{assump:H1:SDEBSDE} hold. Then there exists a unique triple $(X^{t,x},Y^{t,x},Z^{t,x}) \in \cS^p \times \cS^\infty \times \cHBMO$ for any $p\geq 2$ that solves FBSDE \eqref{FBSDE:SDE}-\eqref{FBSDE:BSDE}. Moreover,
\begin{align}
\label{eq:normunibounded}
\sup_{(t,x)\in[0,T]\times \IR^m} \|Y^{t,x}\|_{\cS^\infty}+\|Z^{t,x}\|_{\cHBMO} <\infty.
\end{align}
Furthermore, $Y^{t,x}$ and $Z^{t,x}$ are Markovian: there exist two Borel-measurable functions $u:[0,T]\times \IR^m\to \IR$ and $v:[0,T]\times\IR^m\to\IR^d$ such that 
\[
Y^{t,x}_s=u(s,X^{t,x}_s),\ s\in[t,T],\ \IP\text{-a.s} \qquad Z^{t,x}_s=v(s,X^{t,x}_s),\quad\ \ud s\otimes \IP\text{-a.s.}
\]
\end{prop}
\begin{proof}
Existence and uniqueness follow from standard SDE results and the results of the previous section. 

Let $(t,x)\in[0,T]\times \IR^m$. The uniform bounds of the $\cS^\infty$- and $\cHBMO$-norms, with relation to $t$ and $x$, follow from the uniform bounds of the growth assumption in the driver $f$ and the arguments used to prove Theorem \ref{maintheo:ch3:existence}. 

The constant $K$ appearing in the domination \eqref{assump:FBSDE:growth} is independent of $x$ and $t$ and it is shown in the proof of Theorem \ref{maintheo:ch3:existence} that the uniform bounds for the norms of $Y$ and $Z$ depend only on the constants $K$ and $M$ appearing in Assumptions \ref{assump:H0} and \ref{assump:H1}, which are essentially the same constants appearing in Assumption \ref{assump:FBSDE:growth}.

The Markov property of $(Y,Z)$ is also trivial. It follows from arguments similar to those in Lemma 4.1 in \cite{ElKarouiPengQuenez1997} or Theorem 4.1.1 in \cite{dosreis2011}.
\end{proof}
The next result states the continuous dependence of the solution in the Euclidean parameter $x$.
\begin{coro}
\label{coro:conituityofYinx}
Take $t\in[0,T]$ and $x,x'\in \IR^m$ and let Assumptions \ref{assump:H0:SDE} and \ref{assump:H1:SDEBSDE} hold. Then for any $p\geq 2$ there exists a constant $C>0$ (independent of $x$ and $x'$) such that
\begin{align*}
\IE[\sup_{s\in [t,T]} |Y_s^{t,x}-Y_s^{t,x'}|^p ]
+
\IE\Big[\Big( \int_0^T |Z^{t,x}_s-Z^{t,x'}_s|^2\uds \Big)^\frac{p}2\Big]
\leq C|x-x'|^p
\end{align*}
For $0\leq t\leq s\leq T$ the  mapping $\IR^m  \ni x\mapsto Y^{t,x}_s$ admits a continuous modification where almost all sample paths are $\alpha$-H\"older continuous in $\IR^m$ for any  $\alpha\in(0,1)$.  For $(t,x)\in[0,T]\times\IR^m$ the mapping $s\mapsto Y_s^{t,x}(\omega)$ is continuous for $\IP$-a.s. $\omega\in\Omega$.
\end{coro}
\begin{proof} For simplicity we take $t=0$. The bounds \eqref{eq:normunibounded} play a crucial role because we can now apply the results of Lemma \ref{maincoro:ch3:aprioriestimate} where the involved constants $C$ are independent of the parameters $x$ and $x'$. Then, for $p\geq 2$ we have (recall that $\|Y^x\|_{\cS^\infty}$ is uniformly bounded in $x$)
\begin{align*}
&
\|Y^x-Y^{x'}\|_{\cS^p}^p
\\
&
\quad \leq  C\,\IE\Big[\,\Big(\, \big|g(X^{x}_T)-g(X_T^{x'})\big|
\\
&\hspace{1.8cm}  
+ \int_0^T |f(s,X^x,Y^{x},Z^x)-f(s,X^{x'},Y^x,Z^x)|\uds \Big)^{prq}\big]^{\frac1{qr}}
\\
&
\quad \leq  C\,\big\{ \|X^{x}-X^{x'} \|_{\cS^{pqr}}^{pqr} + 
\\
& \hspace{1.8cm} 
+ \IE\Big[\Big(\sup_{t\in[0,T]}|X^{x}_t-X^{x'}_t|\int_0^T [1+|Z_s^x|+|Z_s^x|^2\,]\uds \Big)^{prq}\Big]^{\frac1{qr}}.
\end{align*}
Using $|z|\leq 1+|z|^2$ along with the fact that $\|Z^x\|_{\cHBMO}$ is uniformly bounded in $x$, the result follows from standard results of SDE\footnote{See for instance Theorem 1.2.5 of \cite{dosreis2011}}, namely that under Assumption  \ref{assump:H0:SDE} we have for any $x,x'\in\IR^m$ and $p\geq 2$ that
\[
\|X^x-X^{x'}\|_{\cS^p}^p=\IE[\sup_{t\in [0,T]} |X_t^x-X_t^{x'}|^p]\leq C_p\,|x-x'|^p,
\]
with $C_p$ a constant independent of $x$ and $x'$.
The continuity of the map in $t$ and $x$ follows from the properties of the solution of the BSDE and Kolmogorov's continuity criterion for $p$ large enough (see Theorem 2 in \cite{Schilling2000}). So we can deduce the result concerning $\|Z^x-Z^{x'}\|_{\cH^{p}}^p$ from \eqref{eq:aprioriestimateofdeltaZ}.
\end{proof}
The previous result hints that it is perhaps possible to make sense of the variation of $Y^x$, i.e.~that $x\mapsto Y^x$ is differentiable (in some sense). 

Using a particular form of the results in \cite{BriandConfortola2008} we can conclude the existence of all partial derivatives (in $x\in\IR^m$) of $(Y^x,Z^x)$. What is then left to argue is the existence of the total derivative of $Y$. We will work under the following assumption
\begin{assump}
\label{assump:H1:SDEBSDE-differentiability}
Let Assumptions \ref{assump:H0:SDE} and \ref{assump:H1:SDEBSDE} hold. $b$, $\sigma$ and $g$ have continuous, uniformly bounded first order spatial derivatives. $f$ is differentiable with continuous derivatives satisfying for some $K>0$ for all $(t,x,y,z)\in[0,T] \times \IR^m \times \IR \times \IR^d$
\begin{align*}
|\nabla_x f(t,x,y,z)|&\leq K(1+|y|+|y|^k|z|+|z|^2),
\\
|\nabla_y f(t,x,y,z)|&\leq K(1+|y|^{k-1}|z|),
\\
|\nabla_z f(t,x,y,z)|&\leq K(1+|y|^k+|z|).
\end{align*}
\end{assump}
Let $t\in[0,T]$, $x\in\IR^m$, 
$\Theta^x=(X^x,Y^x,Z^x)$ be the solution of \eqref{FBSDE:SDE}-\eqref{FBSDE:BSDE}, and define, at least formally, the dynamics\footnote{For $m\in\IN$, $I_m$ represents the $m\times m$-dimensional identity matrix.}
\begin{align}
\label{FBSDE:partialdifferentiatedSDE}
\nabla_x X^x_t&=I_d+\int_0^t \nabla_x b(s,X^x_s)\nabla_x X^x_s\uds
+\int_0^t \nabla_x \sigma(s,X_s^x)\nabla_x X^x_s\udws,
\\
\label{FBSDE:partialdifferentiatedFBSDE}
\nabla_x Y_t^x &= (\nabla_x g)(X^x_T) \nabla_x X^x_T
-\int_t^T \nabla_x Z^x_s\udws
\\
&
\nonumber
\hspace{2cm}
+ \int_t^T \big\langle 
(\nabla f)(s,\Theta_s^x), 
(\nabla_x X_s^x, \nabla_x Y_s^x, \nabla_x Z_s^x)
\big\rangle
\uds.
\end{align}
\begin{prop}
\label{prop:differentiabilitylemma}
Let $x\in \IR^m$, $p\geq 2$ and let Assumption \ref{assump:H1:SDEBSDE-differentiability} hold. Then the mapping $\IR^m \to \cS^p(\IR^m) \times \cS^p(\IR) \times \cH^p(\IR^d)$, $x\mapsto (X^x,Y^x,Z^x)$ is differentiable in the norm topology and the derivative is the unique solution of FBSDE \eqref{FBSDE:partialdifferentiatedSDE}-\eqref{FBSDE:partialdifferentiatedFBSDE} in $\cS^p\times\cS^p\times \cH^p$ for any $p\geq 2$.

In particular, for $x',x\in \IR^m$ we have
\[
\lim_{x'\to x}\big\{ \| \nabla_x Y^{t,x}-\nabla_x Y^{t,x'}\|_{\cS^p}^p 
+
\| \nabla_x Z^{t,x}-\nabla_x Z^{t,x'}\|_{\cH^p}^p 
\big\}=0. 
\]
\end{prop}
\begin{remark}
As in Theorem 2.2 of \cite{AnkirchnerImkellerReis2007} (or Section 3.1 of \cite{dosreis2011}) it is possible to show that there exists a function $\Omega\times[t,T]\times\IR^m\to\IR$ mapping $(\omega,s,x)\mapsto Y^{t,x}_s$ such that for almost all $\omega$, $Y^{t,x}_s$ is continuously differentiable in $x$, i.e. understanding $x\mapsto Y^{t,x}_s(\omega)$ as a function in $C^1(\IR^m,\IR)$!

For such a result one would need to ask that $\nabla g$ satisfies a Lipschitz condition as well a Lipschitz like condition for the derivatives of $f$. Since the argumentation for such a result follows from a straightforward adaptation of the proofs in \cite{AnkirchnerImkellerReis2007} or \cite{dosreis2011}, we refrain from giving the details. 
\end{remark}

The proof of Proposition \ref{prop:differentiabilitylemma} follows a methodology rather standard for this type setting and so it is given in the appendix.

\subsection{Connection to PDEs}

In the classical theory of FBSDE one is able to connect the solution of FBSDE \eqref{FBSDE:SDE}-\eqref{FBSDE:BSDE} to the solution of a certain PDE. Formally, for $0\leq t\leq s\leq T,\ x\in \IR^m$ the FBSDE \eqref{FBSDE:SDE}-\eqref{FBSDE:BSDE}  
 is related to the PDE: 
\begin{align}
\label{eq:thePDE}
\partial_t u(t,x) + \cL u(t,x) + f\big(t, x, u(t,x), (\nabla_x u \,  \sigma)(t,x)\big)&=0,
\\
\nonumber
u(T,x) &= g(x),
\end{align}
with the second order differential operator $\cL$ being the infinitesimal semi-group generator of the Markov process $(X_s^{t,x})_{s\in[t,T]}$ given by
\[
\cL := \sum_{i=1}^{m}b_i(t,x){\partial_{x_i}} + \frac12 \sum_{i,j=1}^m (\sigma\sigma^T)_{i,j}(t,x)  {\partial_{x_i x_j}}.
\]
The relation between the FBSDE and the PDE is given by the identities:
\begin{align}
\label{eq:FBSDEPDErepresentation}
Y^{t,x}_s=u(s,X^{t,x}_s)
\ \ \text{ and }\ \
Z^{t,x}_s=(\nabla_x u\,\sigma)(s,X^{t,x}_s),
\end{align}
where the second relation holds provided the derivatives of $u$ are appropriately defined. Indeed, under Assumptions \ref{assump:H0:SDE} and \ref{assump:H1:SDEBSDE}, we get from  Corollary \ref{coro:conituityofYinx} that 
\begin{align}
\label{eq:uislipschitzinx}
|u(t,x)-u(t,x')|^p=|Y^{t,x}_t-Y^{t,x'}_t|^p\leq C|x-x'|^p,
\end{align}
and hence $u(t,x)$ is a Lipschitz function in its spatial variables. With the results from Proposition \ref{prop:differentiabilitylemma} we have also that $\nabla_x u$ exists and is continuous (in the spatial variable).
\begin{remark}
Particularizing the forward diffusion by assuming $b=0$,  $\sigma=\sqrt{2}$ and $f(t,y,z)=\nu y z$, one obtains the known Burgers' PDE.
\end{remark}
\begin{remark}[Multidimensional case of $f(y,z)=y|z|^2$]
\label{remark:multidimcase}
Related to a particular case of the Theorem \ref{maintheo:ch3:existence}, we make a small remark concerning the PDE representation of such BSDE (when paired with a SDE) in \emph{higher} dimensions. There exists a phenomenon of finite-time gradient blow-up in PDEs: \cite{ChangDingYe1992} (also presented as Theorem III.6.14 in \cite{Struwe2008}) consider mappings $u$ from the closed unit disk in $\IR^{2}$ into the unit sphere in $\IR^{3}$ which satisfy
\begin{align}
\label{blowup}
\partial_t u = \triangle u + |\nabla_x u|^{2}u, \quad u(0,x) = u_{0}(x), \quad 
u(t,\cdot)\big|_{\partial D^{2}}  = u_{0}\big|_{\partial D^{2}}. \end{align}
They  show that for some smooth and bounded boundary condition $u_{0}$, the solution of (\ref{blowup}) blows up in finite time, i.e.,  the maximal existence interval $[0,T)$ has a finite $T$.

Moreover, in \cite{BertschDalPassovanderHout2002} for the same PDE
\begin{align*}
& \partial_t u = \triangle u + |\nabla_x u|^{2}u \text{ in } \cQ:= \Omega\times\IR^+\\
& u=u_ 0 \text{ on } \partial \Omega\times \IR^+ \quad \text{and }\quad 
 u(0,x) = u_{0}(x) \text{ for }x\in\Omega,
\end{align*}
with $\Omega$ a bounded domain in $\IR^n$ with smooth boundary, the authors show that $u_0\in C^\infty$ exists such that an infinite number of solutions to the PDE exist.
\end{remark}

\subsection{Viscosity solution}

As far as the relation between PDE and FBSDE go, the natural counterpart of the FBSDE solution is a PDE solution in the viscosity sense via \eqref{eq:FBSDEPDErepresentation}. Observe that we have not mentioned in any way an ellipticity assumption for $\sigma$, which is consistent with viscosity theory. Although we have a unique solution for the FBSDE, the FBSDE is only to provide existence for the PDE while uniqueness of the PDE has to be shown via a comparison result for PDE. It is unclear whether such a result holds in this framework that allows for cross term nonlinearities like $yz$ since this type of nonlinearities implies that the Hamiltonian arising from \eqref{eq:thePDE} is not a \emph{proper} function.  Nonetheless, such a comparison result exists in the probabilistic framework (see our Theorem \ref{maintheo:ch3:comparison}). 

We do not give many details below, but as expected, one is able to prove that the $Y$ component of the solution to the FBSDE \eqref{FBSDE:SDE}-\eqref{FBSDE:BSDE} is indeed a viscosity solution to PDE \eqref{eq:thePDE}. 
\begin{theo}
\label{theo:proofofviscosity}
Take $(t,x)\in[0,T]\times \IR^m$ and let  Assumptions \ref{assump:H0:SDE} and \ref{assump:H1:SDEBSDE} hold. Then the mapping $u:[0,T]\times \IR^m \to \IR$ defined by $u(t,x):=Y^{t,x}_t$ is a uniformly bounded continuous function of its temporal and spatial variables and a viscosity solution of \eqref{eq:thePDE}.
\end{theo}
The proof of this theorem is the same as that of Theorem 3.2 in \cite{Pardoux1999} except for the application of the  comparison theorem in the last step of the proof (where we use our Theorem \ref{maintheo:ch3:comparison}). We postpone the proof to the appendix.

\section{PDE Perturbation and large deviations}

In this section we look at the relation between FBSDE and PDE perturbation. In a first approach, the form of the perturbation is quite amenable to a in-depth analysis including the possibility of showing a large deviations principle for the solution of the FBSDE as the perturbation vanishes. The second step looks at another setting of PDE perturbation and only partial results, within the framework of the previous sections, are obtained. We then stretch the theory to recover some other known results on a certain perturbed transport PDE and the Navier-Stokes equations for incompressible fluid flow.

In this section we establish a large deviations principle for the  FBSDE \eqref{FBSDE:SDE}, \eqref{FBSDE:BSDE}. This type of result has its relevance in the theory of perturbation of PDEs. We discuss two types of perturbations, while the first leads to classical results (and relates to the theory of $\varepsilon$-vanishing viscosity solutions) the second shows the limitations of the techniques used.

\subsection{Canonical perturbation in FBSDE framework}

Throughout let $x\in\IR^m$, $0\leq t\leq s\leq T$ and $\varepsilon\in(0,1]$. Moreover, we denote by $C$ a constant that may change from line to line but it is \emph{always independent} of the parameters $\varepsilon$, $x$ or $t$. 

Throughout we work with Assumptions \ref{assump:H0:SDE} and \ref{assump:H1:SDEBSDE} and from the context it is clear that the constants involved in the growth bounds are \emph{independent} of $\varepsilon$! At a later stage we will discuss the situation $|f(y,z)|\leq K|y|\, |z|/\varepsilon$.

We work with the family of equations indexed in $\varepsilon$ and following the dynamics
\begin{align}
\label{eq:perturbedSDE}
X_s^{t,x,\varepsilon}=x+\int_t^s b(r,X_r^{t,x,\varepsilon})\ud r + \sqrt\varepsilon \int_t^s \sigma(r, X_r^{t,x,\varepsilon})\udw_r.
\end{align}
It is known that as $\varepsilon$ vanishes, the solution $X^{t,x,\varepsilon}$ converges to $X^{t,x,0}$ (see Lemma \ref{lemma:propertiesofperturbedSDE} below), which solves the deterministic equation
\begin{align}
\label{eq:deterministicSDE}
X^{t,x,0}_s=x+\int_t^s b(r,X^{t,x,0}_r)\ud r.
\end{align}
Moreover, under some extra assumptions, that the law induced by $X^{t,x,\varepsilon}$ satisfies a large deviations principle (LDP) (see Definition \ref{def:LDPprinciple} below).

We want to show that the FBSDE with dynamics
\begin{align}
\label{eq:perturbedBSDE}
Y_s^{t,x,\varepsilon}=g(X_T^{t,x,\varepsilon})
-\int_s^T Z_r^{t,x,\varepsilon}\udw_r
+\int_s^T f(r, X_r^{t,x,\varepsilon}, Y_r^{t,x,\varepsilon}, Z_r^{t,x,\varepsilon})\ud r,
\end{align}
converges to the deterministic backward equation
\begin{align}
\label{eq:deterministicBSDE}
Y_s^{t,x,0}=g(X_T^{t,x,0})
+\int_s^T f(r, X_r^{t,x,0}, Y_r^{t,x,0}, 0)\ud r,
\end{align}
and that the law induced by $Y^{t,x,\varepsilon}$ satisfies an LDP as well. 

We start with a small result on the properties of $X^{t,x,\varepsilon}$.
\begin{lemma}
\label{lemma:propertiesofperturbedSDE}
Under Assumption \ref{assump:H0:SDE} and for any $\varepsilon\in(0,1]$ equations \eqref{eq:perturbedSDE} and \eqref{eq:deterministicSDE} have unique solutions, $X^{t,x,\varepsilon}\in\cS^p$ for any $p\geq 2$ and $X^{t,x,0}\in C^0([t,T]\times\IR^m,\IR^m)$. Moreover, for any $p\geq 2$, for some $C_p>0$ it holds that
\[
\sup_{\varepsilon\in(0,1]}\|X^{t,x,\varepsilon}\|_{\cS^p}^p\leq C_p(1+|x|^p)\ \text{ and } \ 
\|X^{t,x,\varepsilon}-X^{t,x,0}\|_{\cS^p}^p\leq C_p (1+|x|^p)\varepsilon^p.
\]
Furthermore, $\lim_{\varepsilon\to 0} X^{t,x,\varepsilon}=X^{t,x,0}$ a.s.
\end{lemma}
\begin{proof}
These results are quite standard from either SDE or ODE theory and follow from the good properties of the $b$ and $\sigma$ functions. See e.g. \cite{DemboZeitouni1993}.
\end{proof} 
We next establish properties of  $(Y^{t,x,\varepsilon},Z^{t,x,\varepsilon})$ and $Y^{t,x,0}$ .
\begin{prop}
Let Assumptions \ref{assump:H0:SDE} and \ref{assump:H1:SDEBSDE} hold. Then, a unique solution $(X^{t,x,\varepsilon},Y^{t,x,\varepsilon}, Z^{t,x,\varepsilon})$ of FBSDE \eqref{eq:perturbedSDE},   
\eqref{eq:perturbedBSDE} in  $\cS^p\times \cS^\infty\times \cHBMO$ for any $p\geq 2$ exists. Moreover,
\[
\sup_{\varepsilon \in (0,1]}\ \sup_{(t,x)\in[0,T]\times\IR^m}\  \big\{ 
\|Y^{t,x,\varepsilon}\|_{\cS^\infty} 
+ \|Z^{t,x,\varepsilon}\|_{\cHBMO}
\big\}
<\infty.
\]
The mapping $(t,x)\mapsto Y^{t,x,\varepsilon}$ has a continuous modification.

Additionally, the ODE \eqref{eq:deterministicBSDE} has a unique solution $Y^{t,x,0}\in C^0_b([t,T]\times\IR^m,\IR)$. Moreover, there exists a constant $C>0$ such that
\begin{align}
\label{eq:uzeroislipschitzinx}
|Y^{t,x,0}_t-Y^{t,x',0}_t|\leq C|x-x'|,\qquad t\in[0,T],\ x,x'\in\IR^m.
\end{align}
\end{prop}
\begin{proof}
Existence and uniqueness of a solution to \eqref{eq:perturbedBSDE} for $\varepsilon \in (0,1]$ follow from Theorems \ref{maintheo:ch3:existence} and \ref{maintheo:ch3:uniqueness}. An argumentation similar to that of the proof of \eqref{eq:normunibounded} yields the uniform bounds in $\varepsilon$. The continuity result follows from Corollary \ref{coro:conituityofYinx}.

Concerning $Y^{t,x,0}$ from \eqref{eq:deterministicBSDE}, we easily see that $g$ is still a uniformly bounded function and  $F(s,y)=f(s,X^{t,x,0}_s,y,0)$ for $(s,y)\in[t,T]\times \IR$ is an $\cF_s$-adapted continuous function (in $s$ and $y$) satisfying (simply use \eqref{assump:FBSDE:growth} and \eqref{assump:FBSDE:lip}) a linear growth and Lipschitz condition (in the $y$-variable). After a time reversal, standard ODE theory yields (e.g. Picard-Lindel\"of's theorem) the existence and uniqueness of a continuous solution $Y^{t,x,0}$ to \eqref{eq:deterministicBSDE}. One then obtains  that $Y^{t,x,0}$ is uniformly bounded. 

The estimate for the difference $(Y^{t,x,0}-Y^{t,x',0})$ follows easily from the same estimate for $(X^{t,x,0}-X^{t,x',0})$, Gronwall's inequality combined with the uniform boundedness of $Y^{t,x,0}$ and \eqref{assump:FBSDE:lip}.
\end{proof}
The next result states the convergence of $(Y^{t,x,\varepsilon},Z^{t,x,\varepsilon})$ to $(Y^{t,x,0},0)$ as $\varepsilon$ vanishes. Notice that we can always interpret the deterministic equation \eqref{eq:deterministicBSDE} as a FBSDE with solution $(Y^{t,x,0},0)$.
\begin{prop}
\label{prop:convofYepsilontoY}
Let Assumptions \ref{assump:H0:SDE} and \ref{assump:H1:SDEBSDE} hold. Then for any $p\geq 2$ there exists a constant $C_p>0$ such that
\begin{align*}
\sup_{(t,x)\in[0,T]\times \IR^m}\Big\{\,\IE\big[\sup_{t\leq s\leq T} |Y_s^{t,x,\varepsilon} - Y_s^{t,x,0}|^p\big] + \IE\Big[ \Big(\int_t^T|Z_r^{t,x,\varepsilon}|^2\ud r\Big)^p \Big]\,\Big\} \leq C_p\varepsilon^p.
\end{align*}
We emphasize that $C_p$ is independent of $t$, $x$ or $\varepsilon$.
\end{prop}
\begin{proof} Let $p\geq 2$. We start by interpreting the pair $(Y^{t,x,0},0)$ as the solution of a deterministic BSDE. This interpretation allows us to use Corollary \ref{maincoro:ch3:aprioriestimate} along with the uniform boundedness of $Y^{t,x,\varepsilon}$ and the assumptions to obtain
\begin{align*}
& \|Y^{t,x,\varepsilon}-Y^{t,x,0}\|_{\cS^p}^p
\leq C_\beta
\|X^{t,x,\varepsilon}-X^{t,x,0}\|_{\cS^{\beta p}}^p\leq C_p\,\varepsilon^p,
\end{align*}
where $\beta$ is related to the uniformly bounded (in $\varepsilon$) BMO norm of $Z^{t,x,\varepsilon}*W$ and the result follows from Lemma \ref{lemma:propertiesofperturbedSDE}.

From \eqref{eq:aprioriestimateofdeltaZ} and the results already proved we easily get the sought result concerning $\|Z^{t,x,\varepsilon}-Z^{t,x,0}\|_{\cH^p}^p=\|Z^{t,x,\varepsilon}\|_{\cH^p}^p$.
\end{proof}

\subsubsection*{The corresponding perturbed PDE problem}
By solving \eqref{eq:perturbedSDE} and \eqref{eq:perturbedBSDE} one is able to solve, at least in the viscosity sense, the corresponding PDE, namely in view of \eqref{eq:thePDE} and Theorem \ref{theo:proofofviscosity} one has for $(t,x)\in[0,T]\times \IR^m$
\begin{align*}
\partial_t u^\varepsilon (t,x)+
\cL^{\varepsilon} u^\epsilon(t,x) + f\big( t,x,u^\varepsilon(t,x),\sqrt{\varepsilon}(\nabla_x u^{\varepsilon} \sigma)(t,x) \big)&=0,\\
  u^\varepsilon(T,x)&=g(x),
\end{align*}
with $\cL^\varepsilon := (b\cdot\nabla )(t,x) + \frac{\varepsilon}2 \sum_{i,j=1}^m (\sigma\sigma^T)_{i,j}(t,x)  {\partial_{x_i x_j}}$. And as $\varepsilon$ vanishes one is lead to the first order hyperbolic PDE for $(t,x)\in[0,T]\times\IR^m$
\begin{align*}
\partial_t u^0 (t,x)+ (b\cdot\nabla u^0)(t,x)
 + f\big( t,x,u^0(t,x),0 \big)=0,\quad u^0(T,x)=g(x),
\end{align*}
or after a time inversion, $t\mapsto T-t$, with $(t,x)\in[0,T]\times\IR^m$
\begin{align*}
\partial_t u^0 (t,x)= (b\cdot\nabla u^0)(t,x)
 + f\big( t,x,u^0(t,x),0 \big),\quad u^0(0,x)=g(x).
\end{align*}

\subsection{The large deviations result}

We now explain the large deviations principle (LDP). For a deeper analysis on large deviations we point the reader to \cite{FreidlinWentzell1998}, \cite{DemboZeitouni1993},  \cite{FengKurtz2006} and references therein. In this section we extend the results of \cite{DossRainero2007} and \cite{Rainero2006} to another class of FBSDE.

We denote $C^0([a,b],\IR^m)$ (for $0\leq a\leq b<+\infty$) the set of continuous functions on the interval $[0,T]$ with values in $\IR^m$ and we consider in this space the uniform norm
\begin{align}
\label{uniformmetric} \rho_{[a,b]}(\phi):=\sup_{t\in[a,b]}|\phi(t)|.
\end{align}
\begin{defi}
\label{def:LDPprinciple}
The family of processes $(X^{\varepsilon}_t)_{t\in[0,T]}$ depending on a parameter $\varepsilon$ is said to satisfy a \emph{large deviations principle} (LDP) with rate function $S(\Psi)$ if the following conditions hold for every Borel set $A\subset C([0,T])$
\begin{align*}
&\limsup_{\varepsilon\to0} \varepsilon\log\IP[X^\epsilon\in A]\leq \inf_{\Psi\in \text{Cl}(A)} S(\Psi)
\\
&\liminf_{\varepsilon\to0} \varepsilon\log\IP[X^\epsilon\in A]\geq - \inf_{\Psi\in \text{Int}(A)} S(\Psi),
\end{align*}
where Cl\ \!$(A)$ denotes the closure of the set $A$ and Int\ \!$(A)$ denotes the interior of the set $A$.
\end{defi}
Before we state that $X^{t,x,\varepsilon}$ satisfies an LDP, we recall the Cameron-Martin space $H$ of absolutely continuous functions.
\begin{defi}[Cameron-Martin space] Let $k\in\IN$, the space $H([t,T],\IR^k)$ is defined as the space of continuous functions $\phi\in C^0([t,T],\IR^k)$ for which there exists a square integrable function $\dot\phi$ such that $\phi_s= \phi_t+\int_t^s \dot\phi_r\ud r$ for $0\leq t\leq s\leq T$ (i.e. $\phi$ is absolutely continuous). 
\end{defi}
We now make a new assumption, stronger than Assumption \ref{assump:H0:SDE}, for the result stating that $X^{t,x,\varepsilon}$ satisfies an LDP.
\begin{assump}
\label{assump:LDPforX}
Let Assumption \ref{assump:H0:SDE} hold. $b$ and $\sigma$ satisfy $b(t,x)=b(x)$ and $\sigma(t,x)=\sigma(x)$ for all $(t,x)\in[0,T]\times \IR^m$ and are uniformly bounded.
\end{assump}
\begin{theo}[Theorem 5.6.7 of \cite{DemboZeitouni1993}]
\label{theo:LDPforXepsilon}
Let Assumption \ref{assump:LDPforX} hold. Then the solution  $X^{t,x,\varepsilon}$ of \eqref{eq:perturbedSDE} satisfies, as $\varepsilon$ goes to zero, a large deviations principle in the space $C^0([t,T],\IR^m)$ associated to the rate function $I_x$ defined for any $\phi\in C^0([t,T],\IR^m)$ by
\begin{align*}
I_x(\phi)&:=\inf\Big\{
\frac12\int_t^T|\dot{v}_r|^2\ud r,\ v\in H([t,T],\IR^d)\ \text{ such that }
\\
&
\hspace{2.4cm} \phi_s=x+\int_t^s b(\phi_r)\ud r + \int_t^s \sigma(\phi_r)\dot{v}_r\ud r,\ s\in[t,T]
\Big\},
\end{align*}
with the convention that $\inf \varnothing =+\infty$.
\end{theo}

\begin{remark}
\label{remark:Qoperator}
We point out that if $\sigma\sigma^T$ is invertible, then the rate function functional can be written in a simplified manner, more precisely 
\begin{equation*}
I_x(\phi):=\left\{ \begin{array}{cll}
\frac{1}{2} \int_t^T \cQ^*_{\phi_r}[\dot\phi_r-b(\phi_r)] \ud r
&,& \text{if } \phi\in H([t,T],\IR^m)\text{ and }\phi_t=x\\
+\infty&,&\text{otherwise,}
\end{array} \right.
\end{equation*}
where $\cQ^*$ is defined by $\cQ^*_u[v]:=\langle v , (\sigma\sigma^T)^{-1}(u)v \rangle$ for all $u,v\in\IR^m$ .
\end{remark}
We now define a set of operators from and onto the space of continuous functions that, combined with the contraction principle\footnote{In rough, let $\Theta$ and $\Xi$ be two complete separable metric spaces and $f_\varepsilon:\Theta\mapsto \Xi$ a family of continuous functions in those metric spaces. If $\lim_{\varepsilon \to 0}f_\epsilon=f$ exists uniformly over compact sets, then the contraction principle states that if the $\Theta$ valued process $X_\epsilon$, satisfies an LDP with rate function $I$ then $Y_\epsilon=f_\varepsilon(X_\epsilon)$ also satisfies an LDP with the same rate and with a raw rate function $J(y)=I\big(f^{(-1)}(y)\big)$. Moreover, if $I$ is a good rate function then so is $J$. See Section 2 of \cite{Varadhan1984} or Section 4.2 of \cite{DemboZeitouni1993}.} (see Theorem 2.4 in \cite{Varadhan1984} or Theorems 4.2.1 and 4.2.21 of \cite{DemboZeitouni1993}), will allows us to conclude that $Y^{t,x,\varepsilon}$ satisfies an LDP.
\begin{defi}
The operator $F^\varepsilon$ is defined as
\begin{align*}
F^\varepsilon: C^0([t,T], \IR^m)&\to C^0([t,T], \IR)
\\
\phi &\mapsto F^\varepsilon(\phi) = u^{\varepsilon}(\cdot,\phi_\cdot), 
\end{align*}
where for $\varepsilon\in(0,1]$ the function $u^{\varepsilon}$ is defined by $u^\varepsilon(t,x):=Y^{t,x,\varepsilon}_t$ the first component of the solution to \eqref{eq:perturbedBSDE}. And for $\varepsilon=0$, $u^0(t,x)$ is given by $Y^{t,x,0}_t$ solution to \eqref{eq:deterministicBSDE}.
\end{defi}
We observe that for all $0\leq t\leq s\leq T$, $x\in \IR^m$ and  $\varepsilon\in[0,1]$ we have $Y^{t,x,\varepsilon}_s = F^\varepsilon( X^{t,x,\varepsilon})(s) = u^\varepsilon(s, X^{t,x, \varepsilon}_s )$ (see Theorem \ref{theo:FBSDEexistenceanduniqueness}). We now state and prove the result concerning the LDP satisfied by the law induced by $Y^{t,x,\varepsilon}$.
\begin{theo}
\label{theo:LDPforY}
Let Assumption \ref{assump:LDPforX} and \ref{assump:H1:SDEBSDE} hold. Then the process $Y^{t,x,\varepsilon}$ satisfies an LDP in $C^0([t,T],\IR)$ with rate function $\widehat{I}_x$ defined for any $\psi\in C^0([t,T], \IR)$ by
\begin{align*}
\widehat{I}_x(\psi)
&:=
\inf\Big\{
I_x(\phi): \phi\in H([t,T],\IR^m) \text{ such that } \psi=F^0(\phi)
\Big\},
\end{align*}
with the convention that $\inf \varnothing =+\infty$ and where $I_x(\cdot)$ is the rate function from Theorem \ref{theo:LDPforXepsilon}.
\end{theo}
\begin{remark}
If the mapping $\sigma\sigma^T$ is invertible then the rate function $\widehat I_x$ can be expressed in an alternative fashion using the operator $\cQ^*$ introduced in Remark \ref{remark:Qoperator}, namely for $\psi\in C^0([t,T],\IR)$ and $x\in\IR^m$
\begin{align*}
\widehat I_x(\psi) &= \inf\Big\{\, \frac12 \int_t^T \cQ^*_{\phi_r}[\dot\phi_r-b(\phi_r)]\ud r,\ \text{ over }\\
&\hspace{2.2cm} \phi\in H([t,T],\IR^m),\ \phi_t=x,\ \psi_s=u^0(s,\phi_s),\ s\in[t,T]
\, \Big\}.
\end{align*}
\end{remark}
\begin{proof}[Proof of Theorem \ref{theo:LDPforY}] 
This proof uses the well known contraction principle, see Theorem 2.4 in \cite{Varadhan1984}. To apply the contraction principle we need only to show that  $F^\varepsilon$ for $\varepsilon\in[0,1]$ are uniformly continuous operators from $C^0([t,T],\IR^m)$ onto $C^0([t,T],\IR)$, and moreover, that $F^\varepsilon$ converges uniformly to $F^0$ over all compact sets of $C([t,T],\IR^m)$ as $\varepsilon$ vanishes. Notice from Corollary \ref{coro:conituityofYinx} that the mappings $x\mapsto Y^{t,x,\varepsilon}_s$ and $s\mapsto Y^{t,x,\varepsilon}_s$ are  continuous ($\IP$-a.s.).

{\sf Step 1 - Continuity of $F^\varepsilon$:} Let $\varepsilon\in[0,1] $. From \eqref{eq:uislipschitzinx} and \eqref{eq:uzeroislipschitzinx} we have that $x\mapsto u^\varepsilon(t,x)$ is Lipschitz continuous. Let $\cK^0$ be a compact subset of $C^0([t,T],\IR^m)$ and let $(\phi^n)_{n\in\IN}$ be a sequence of functions from $\cK^{0}$ that converge to some $\phi\in \cK^{0}$ in the uniform norm, i.e. $\lim_{n\to\infty} \rho_{[t,T]}(\phi^n-\phi)=0$, where $\rho$ is defined in \eqref{uniformmetric}. From the definition of $F^\varepsilon$, \eqref{eq:uislipschitzinx}, \eqref{eq:uzeroislipschitzinx} and the continuity of $Y^{t,\cdot}_\cdot$ we have 
\begin{align*}
\big[\rho_{[t,T]}\big(F^\varepsilon(\phi^n)-F^\varepsilon(\phi)\big)\big]^2 
&= \sup_{s\in[t,T]}|u^\varepsilon(s,\phi^n_s)-u^\varepsilon(s,\phi_s)|^2
\\
&
= \sup_{s\in[t,T]}|Y^{s,\phi^n_s,\varepsilon}_s-Y^{s,\phi_s,\varepsilon}_s|^2
\\ &
\leq C \sup_{s\in[t,T]}|\phi^n_s-\phi_s|^2 = C\big[\rho_{[t,T]}(\phi^n-\phi)\big]^2\stackrel{n\to\infty}{\longrightarrow} 0.
\end{align*}
Note that the constant $C$ present in the estimate of Corollary \ref{coro:conituityofYinx} is independent of $x$ or $\phi$.

{\sf Step 2 - Convergence over compacts:} To prove that $F^\varepsilon$ converges to $F^0$ uniformly over the compact sets $\cK^0$ of $C^0([t,T],\IR^m)$ we start by defining the set $\cI:=\{\phi_s: \phi\in\cK^0,\ s\in[t,T]\}$, as the set of points consisting of the images of the $\phi \in \cK^0$ in $\IR^m$ for all $s\in[t,T]$. Remark that $\cI$ is a compact set of $\IR^m$.

Now, by Proposition \ref{prop:convofYepsilontoY} (in the particular case of $s=t$) and taking the set $\cI$ into account, there exists a constant $C>0$ such that 
\begin{align*}
&
\sup_{\phi\in\cK^0}\Big[
\rho_{[t,T]}\big(F^{\varepsilon}(\phi)-F^{0}(\phi)\big)\Big]^2
\\
&\quad =\sup_{\phi\in \cK^0} \sup_{s\in[t,T]}
|u^\varepsilon(s,\phi_s)-u^0(s,\phi_s)|^2
=\sup_{\phi\in \cK^0} \sup_{s\in[t,T]}
|Y^{s,\phi_s,\varepsilon}_s-Y^{s,\phi_s,0}_s|^2
\\
&
\quad 
\leq \sup_{x\in \cI} \sup_{s\in[t,T]}
|Y^{s,x,\varepsilon}_s-Y^{s,x,0}_s|^2\leq C\,\varepsilon^2 \to 0\ \ \text{ as }\  \varepsilon \to 0.
\end{align*}
The result now follows from the already mentioned contraction principle.
\end{proof}

\subsection{PDE perturbation and Burgers' equation with damping}
\label{subsec:burgerswithdamping}

Here we want to discuss the following PDE for $\lambda>0$, $a\in\IR$, $x\in\IR$ where $u:[0,T]\times \IR\to \IR$ and $g\in C^0_b$
\begin{align}
\label{eq:burgersperturbationPDE}
\partial_t u^{\varepsilon} + a u^{\varepsilon} \nabla_x u^{\varepsilon} + \lambda u^{\varepsilon} = \varepsilon 
\Delta u^{\varepsilon} 
,\quad u^{\varepsilon}(0,x)=g(x).
\end{align}
Or, after a canonical time inversion $t\mapsto T-t$ for $t\in[0,T]$
\begin{align}
\label{eq:burgersbackwardsperturbationPDE}
\partial_t u^{\varepsilon} + \varepsilon \Delta 
 u^{\varepsilon}=a u^{\varepsilon} \nabla_x u^{\varepsilon} + \lambda u^{\varepsilon}, \quad u^{\varepsilon}(T,x)=g(x).
\end{align}
Via the already discussed Feynman-Kac formula we can write (at least formally) the corresponding FBSDE system for $x\in\IR$ and $0\leq t\leq s\leq T$
\begin{align}
\label{eq:correspondingburgerBDSE}
X_s^{t,x,\varepsilon}&=x+ \sqrt{2 \varepsilon}\, (W_s-W_t),
\\
\nonumber
Y_s^{t,x,\varepsilon}&=g(X_T^{t,x,\varepsilon})
-\int_s^T Z_r^{t,x,\varepsilon}\udw_r
-\int_s^T\Big[ \frac{a\,}{\sqrt{2\varepsilon}}\, {Y_r^{t,x,\varepsilon} Z_r^{t,x,\varepsilon}}  + \lambda Y_r^{t,x,\varepsilon} \Big]\ud r.
\end{align}
The appearance of the $1/\sqrt{2\varepsilon}$ weight is due to the representation $Z=(\nabla_x u \sigma)(\cdot,X)$ with $\sigma=\sqrt{2\varepsilon}$; see \eqref{eq:FBSDEPDErepresentation}. This type of FBSDE still fits in the setting we discussed in the first sections, but the presence of the term $1/\varepsilon$ creates complications if we consider the FBSDE not just for one single fixed $\varepsilon$.

For each fixed $\varepsilon\in(0,1]$ we obtain from Theorem \ref{theo:FBSDEexistenceanduniqueness} existence and uniqueness of the triple $(X^{t,x,\varepsilon}, Y^{t,x,\varepsilon}, Z^{t,x,\varepsilon})$ solving the FBSDE above in $\cS^p\times \cS^\infty\times \cHBMO$ for any $p\geq 2$. The real interesting result is the uniform boundedness in $\varepsilon$ of $\| Y^{t,x,\varepsilon} \|_{\cS^\infty}$.
\begin{lemma}Let $(t,x)\in[0,T]\times\IR$ then in the above defined framework
\[
\sup_{x\in\IR,\,\varepsilon\in(0,1],\,t\in[0,T]}\| Y^{t,x,\varepsilon} \|_{\cS^\infty}<\infty.
\]
\end{lemma}
\begin{proof}
From the original existence Theorem \ref{maintheo:ch3:existence} and the lemmata used in its proof, we easily obtain for each fixed $\varepsilon\in(0,1]$ that 
\[
\sup_{x\in\IR,\,t\in[0,T]}\| Y^{t,x,\varepsilon} \|_{\cS^\infty}<\infty.
\]
We sketch the argument concerning the uniform boundedness on $\varepsilon$ since it is based on an argument we have already seen in the proof of Lemma \ref{lemma:exist:uniq:ofYnZn}. We again use the linearization and measure change argument. Define processes $e_t=\exp\{-\lambda t\}$ and $b^\varepsilon_t := a Y^\varepsilon_t/ \sqrt{2\varepsilon}$. Since $Y^\varepsilon$ is a bounded process it is clear that $b^\varepsilon*W\in BMO$ and we can define a new probability measure $\IQ^\varepsilon$ with Radon-Nikodym density $\ud \IQ^\varepsilon/\ud \IP=\cE(b^\varepsilon*W)$, with relation to which $W^\varepsilon_\cdot=W_\cdot-\int_0^\cdot b^\varepsilon_r\ud r$ is a Brownian motion.

Applying It\^o's formula to $e_t Y^\varepsilon_t$ and changing to the measure $\IQ^\varepsilon$ we obtain 
\begin{align*}
e_s Y_s^{t,x,\varepsilon}&=e_T  g(X_T^{t,x,\varepsilon}) - \int_s^T e_r Z^{t,x,\varepsilon}_r \udw^\varepsilon_r
\\
&\quad 
\Rightarrow \quad Y_s^{t,x,\varepsilon} =
e^{-\lambda(T-s)} \IE^{\IQ^\varepsilon}[g(X_T^{t,x,\varepsilon}) |\cF_s]
\\
&\quad
\Rightarrow\quad 
\sup_{x\in\IR}\ \sup_{\varepsilon\in(0,1]}\ \big\|\sup_{s\in[t,T]} |Y_s^{t,x,\varepsilon}|\,\big\|_ {L^\infty} 
\leq \sup_{x\in\IR}|g(x)|<\infty,
\end{align*}
where the last line follows from the uniform boundedness of $g$ and from $e^{-\lambda(T-s)}\leq 1$ for all $s\in[0,T]$. 
\end{proof}
This result hints that the boundedness of the solution carries over to the limit (as $\varepsilon\to 0$). On the other hand, obtaining a similar estimate for $Z^{t,x,\varepsilon}$ in the $\cHBMO$-norm is impossible without extra assumptions. 

First remark that the limiting PDE of \eqref{eq:burgersperturbationPDE} as $\varepsilon\to 0$ is the first order equation:
\begin{align*}
\partial_t v + a v\nabla_x v +\lambda v=0,\quad v(0,x)=g(x),\quad (t,x)\in[0,T]\times\IR.
\end{align*}
For this type of transport PDE it is well known that even with smooth initial boundary data the solution can develop shocks in finite time which translates into an explosion of the spatial derivative in finite time, see \cite{Tersenov2010}. An estimate on $\nabla_x u^{\varepsilon}$ (or $Z$) independent of $\varepsilon$ would mean that the derivative remains stable (non-explosive) as we decrease the \emph{regularizing} impact of the perturbation.

\subsubsection*{Getting help from fully coupled FBSDE}

If one assumes that $\lambda\geq K^2a^2T$ then it is possible to prove via \emph{fully coupled} FBSDE arguments that $\sup_{\varepsilon\in(0,1]}\|\nabla_x u^\varepsilon\|_\infty\leq \sqrt{2}\,K$. This condition is, in some cases, close to that of Theorem 1 in \cite{Tersenov2010} and it roughly means that if a strong enough damping is present in the PDE (via the term $\lambda u^\varepsilon$) then the cross term $u^\varepsilon\nabla_x u^\varepsilon$ does not lead to shocks up to time $T$. We remark though, that  \cite{Tersenov2010} works with bounded domains. 

Even with the mentioned condition these authors have not been able to show the result using the \emph{decoupled FBSDE} theory alone. Notice that one could effectively use the theory developed in Proposition \ref{prop:differentiabilitylemma} (with the appropriate assumptions) to show that the solution $(Y^{t,x,\varepsilon}, Z^{t,x,\varepsilon})$ (of \eqref{eq:correspondingburgerBDSE}) is differentiable in $x$, but the mechanics employed does not lead to estimates on the moments of $Z^{t,x,\varepsilon}$ or $\nabla_x Y^{t,x,\varepsilon}$ that are independent of $\varepsilon$! 

An alternative FBSDE representation of PDE \eqref{eq:burgersbackwardsperturbationPDE} follows by interpreting the term $a u^\varepsilon \nabla_x u^\varepsilon$ not as a component of the BSDE's driver but as a part of the second order differential operator, i.e., as the drift of the forward diffusion (compare FBSDE \eqref{FBSDE:SDE}-\eqref{FBSDE:BSDE} and PDE \eqref{eq:thePDE} along with the operator $\cL$). This leads then to a \emph{fully coupled} FBSDE, namely take $(t,x)\in [0,T]\times\IR^m$ and $s\in[t,T]$
\begin{align*}
X_s^{t,x,\varepsilon}
&
=x-\int_t^s a Y^{t,x,\varepsilon}_r \ud r+ \sqrt{2\varepsilon}\int_t^s1\udw_r,
\\
Y_s^{t,x,\varepsilon}
&
=g(X_T^{t,x,\varepsilon})-\int_s^T \lambda Y^{t,x,\varepsilon}_r\ud r-\int_s^T Z_s^{t,x,\varepsilon}\udw_r.
\end{align*}
We give a quick sketch of the arguments for the \emph{fully coupled case} since the theory differs a bit from what has been presented before. Although the argumentation is not complicated, a full justification would be lengthy (for this already long manuscript); for more on the theory of  \emph{fully coupled FBSDE} we point the reader to \cite{LNiM1702}.

Let us assume that the above coupled FBSDE has a unique solution $(X^{t,x,\varepsilon},Y^{t,x,\varepsilon},Z^{t,x,\varepsilon}) \in  \cS^p\times\cS^\infty\times \cHBMO$, and rather trivially (to see) that  $\|Y^{t,x,\varepsilon}\|_{\cS^\infty}$ and $\|Z^{t,x,\varepsilon}\|_{\cHBMO}$ are uniformly bounded in their parameters. 
For $x,x'\in\IR^m$, $t\in[0,T]$ and $\delta Y=Y^{t,x,\varepsilon}-Y^{t,x',\varepsilon}$ and $\delta Z=Z^{t,x,\varepsilon}-Z^{t,x',\varepsilon}$, applying It\^o's formula to $|\delta Y_s|^2$ for $s\in[t,T]$ and noticing that
\[
|g(X_T^{t,x,\varepsilon})-g(X_T^{t,x',\varepsilon})|^2\leq 
2K^2\Big( |x-x'|^2 + a^2(T-t)\int_t^T |\delta Y_s|^2\uds\Big),
\]
 leads to
\begin{align*}
&\int_t^T |\delta Z_s|^2\uds + 
|\delta Y_t|^2
\\
&\quad  \leq 
2K^2|x-x'|^2+2(K^2a^2(T-t)-\lambda)\int_t^T |\delta Y_s|^2\uds 
-\int_t^T 2\delta Y_s \delta Z_s\udws.
\end{align*}
Assuming that $\lambda \geq K^2a^2T$ and taking $\cF_t$-conditional expectation yields
\[
|\delta Y_t|^2=|u^\varepsilon(t,x)-u^\varepsilon(t,x')|^2\leq 2 K^2|x-x'|^2.
\]
Surely, this result does not imply differentiability in $x$, but with some extra arguments (mollification of the terminal condition then passing to the limit) one easily obtains the spatial differentiability of $Y^{t,x,\varepsilon}$ (and hence of $u(t,x)$). From the above estimate, one would obtain $|\nabla_x u^{\varepsilon}| \leq \sqrt{2}K$ uniformly in $\varepsilon$. This again hints that the estimate would hold in the limit.

\subsection{Incompressible Navier-Stokes}

While in the previous subsection we were forced in the end to work with fully coupled FBSDE in order to obtain the extra result we where searching for, we now give an example which simplifies greatly if one uses the decoupled FBSDE instead of the fully coupled ones. In this section we look at the works by \cite{CruzeiroShamarova2009,CruzeiroShamarova2010} where they present a connection between FBSDE and the incompressible Navier-Stokes equation.

More than being a standard PDE, the Navier-Stokes equation describes a conservation law (see the equations below) and it is this conservation law that is the hard part in the FBSDE representation. We saw in Section \ref{section:FBSDEandPDE} that the FBSDE yields solutions to PDE but the  conservation of volume condition (i.e. div$(u)=\langle\1,\nabla_x u\rangle=0$, see \eqref{eq:NS}) has to be verified. One does not obtain it as a side product of the existence of solution.

Following \cite{CruzeiroShamarova2009,CruzeiroShamarova2010} we look at the incompressible Navier-Stokes equation with $u:[0,T]\times \IR^n\to \IR$ satisfying for $(t,x)\in[0,T]\times \IR^n$
\begin{align}
\label{eq:NS}
\partial_t u &= - u \langle\1,\nabla_x u\rangle + \nu \Delta u-\nabla_x p,\ \ \text{div } u=0,
\\
\nonumber
u(0,x)&=h(x),\quad \text{div } h=0,
\end{align}
where $\nu>0$, div$(u):=\langle \1,\nabla_x u\rangle$ and $h\in C^1_b(\IR^n)$. The function $p$ denotes the pressure field which is assumed to satisfy $\nabla_x p=K$ for some constant $K$. Throughout we work with $n=2$ but the results presented here are easily extended to the general case.

In view of the calculations carried out in Subsection \ref{subsec:burgerswithdamping} the FBSDE corresponding to the PDE \eqref{eq:NS} (after a time reversal) is for $(t,x)\in[0,T]\times\IR^2$ 
\begin{align*}
X_s^{t,x}&=x+ \sqrt{2 \nu}\, (W_s-W_t),\quad s\in[t,T]
\\
Y_s^{t,x}&= h(X_T^{t,x})-\int_s^T Z_r^{t,s}\ud r -\int_s^T \big[ K+\frac{1}{\sqrt{2\nu}}Y^{t,x}_r \langle \1,Z_r^{t,x}\rangle\big]\ud r,
\end{align*}
where $W=(W^1,W^2)$, $X=(X^1,X^2)$ and $Z=(Z^1,Z^2)$.

From Theorems \ref{maintheo:ch3:existence}, \ref{maintheo:ch3:uniqueness} and \ref{theo:proofofviscosity} easily follows existence and uniqueness of a solution $(X,Y,Z)$ to the FBSDE in $\cS^p\times \cS^\infty\times \cHBMO$ (for any $p\geq 1$) as well as its connection to the PDE. Moreover since $h\in C^1_b$, Proposition \ref{prop:differentiabilitylemma} also holds, yielding the existence of $(\nabla_x Y^{t,x},\nabla_x Z^{t,x})\in \cS^p\times \cH^p$ for $p\geq 2$. 

One of the difficulties in \cite{CruzeiroShamarova2009} was showing that $Y$ satisfied the free divergence condition, namely that for $u(t,x):=Y^{t,x}_t$ the function $u$ satisfies div$(u)=0$ for all $(t,x)\in[0,T]\times\IR^2$. For the case we present here, this is considerably easier than \cite{CruzeiroShamarova2009}. We have 
\begin{align*}
\nabla_x Y^{t,x}_s&= (\nabla_x h)(X_T^{t,x})\nabla_x X_T^{t,x} - \int_s^T \nabla Z^{t,x}_r\ud W_r
\\
&\qquad 
-\int_s^T \frac{1}{\sqrt{2\nu}} \Big[\nabla_x Y_r^{t,x} \langle \1, Z_r^{t,x}\rangle + Y_r^{t,x} \nabla_x \big(\langle \1, Z_r^{t,x}\rangle\big)
\Big]\ud r,
\end{align*}
where $\nabla_x X_s=(1,1)$. Then, div$(u)=0$ is found by looking at the processes  $U:=\langle\1,\nabla_x Y^{t,x}\rangle$ and $V^i:= \langle \1,\nabla_x Z^{t,x,i}\rangle$ for $i\in\{1,2\}$. We have then
\begin{align*}
U_s = \big(\text{div}(h)\big)(X_T^{t,x})-\int_s^T V_r\ud W_r -\int_s^T\Big[
\frac{\langle\1,Z_r^{t,x}\rangle}{\sqrt{2\nu}} U_r + 
\frac{Y_r^{t,x}}{\sqrt{2\nu}} \langle\1,V_r\rangle
\Big]
\ud r.
\end{align*}
Now since div$(h)=0$ and using the linearization and measure change argument, i.e. for $s\in[t,T]$ define  $b_s:={Y_s^{t,x}}/{\sqrt{2\nu}}$, a probability measure $\IQ$ with Radon-Nikodym density $\cE(b*W)$ 
and $e_s:=\exp\{\int_t^s \langle\1,Z_r^{t,x}\rangle/\sqrt{2\nu}\ud r\}$. Note that since $Y\in \cS^\infty$ and $Z\in\cHBMO$ the processes $b$ and $e$ are well defined, see for instance the arguments used in the proof of Theorem \ref{maintheo:ch3:uniqueness}.

We obtain then
\[
U_s = -e_s^{-1}\int_s^T e_r V_r\ud W_r^\IQ
\ \ \Rightarrow \ \ 
U_s=0\ \ \IQ \text{-a.s. for all }s\in[t,T].
\]
In particular, $U_t=\text{div}(Y^{t,x}_t)=0$ as we sought. This proves that the solution of the FBSDE provides a mild solution to the Navier-Stokes equation \eqref{eq:NS}. 
\vspace{0.5cm}

{\bf Acknowledgment:} We thank Peter Friz and Peter Imkeller for their helpful comments. Christoph Frei gratefully acknowledges financial support by the Natural Sciences and Engineering Research Council of Canada through grant 402585. Gon{\c c}alo dos Reis was partially supported by the CMA/FCT/UNL, under the project PEst-OE/MAT/UI0297/2011.

\appendix
\section{Proof of Proposition \ref{prop:differentiabilitylemma}}
\begin{proof}[Proof of Proposition \ref{prop:differentiabilitylemma}]
The results concerning the forward diffusion are standard, see for instance Section 1.2.4 of \cite{dosreis2011}.

This proof has 3 steps. We first establish a candidate for the said derivative process, then we prove that the candidate is the correct one and we finish by showing continuity of the derivatives in the spatial component. Throughout let $j\in\{1,\cdots,m\}$ and $t\in[0,T]$.

{\sf Step 1 -  The candidate for the partial derivatives:}
A quick look at the BSDEs composing the lines of \eqref{FBSDE:partialdifferentiatedFBSDE} shows that the lines are independent of each other. So, in view of the assumptions and the already obtained results concerning $Y^x$ and $Z^x$ it easily follows that for \emph{each line} of the system \eqref{FBSDE:partialdifferentiatedFBSDE} Assumptions A1--A4 of \cite{BriandConfortola2008} hold. Hence we can apply their results to conclude that \eqref{FBSDE:partialdifferentiatedFBSDE} has a unique solution $(\nabla_{x_j} Y^x,\nabla_{x_j} Z^x)\in\cS^p \times \cH^p$ for all $p\geq 2$ and $j=\{1,\cdots,m\}$. These processes are the natural candidates for the $j$-th partial derivatives of $Y^x$ and $Z^x$.

{\sf Step 2 -  The candidates are the correct ones:} We now need only to prove that the solution of \eqref{FBSDE:partialdifferentiatedFBSDE} corresponds indeed to the partial derivative of $(Y^x,Z^x)$. To this end we need only to show that for any $h>(0,1]$ such that we have for all $p\geq 2$:
\[
\lim_{h\to 0} \bigg\{
\Big\|
\frac{Y^{x+he_j}-Y^x}{h}-\nabla_{x_j}Y^x
\Big\|_{\cS^p}
+
\Big\|
\frac{Z^{x+he_j}-Z^x}{h}-\nabla_{x_j}Z^x
\Big\|_{\cH^p}
\bigg\}=0.
\]
Again we do not give the details, the proof of this result follows from a straightforward combination of the mechanics employed in the proof of Corollary \ref{maincoro:ch3:aprioriestimate} and the techniques used to prove Theorem 2.1 of \cite{AnkirchnerImkellerReis2007} (alternatively see Section 3 of \cite{dosreis2011}).

{\sf Step 3 -  Total differentiability and the continuity estimate:} After having established that the partial derivatives exist and satisfy good integrability properties, it remains only to argue in favor of total differentiability of the map $x\mapsto Y^x$. This will be done by showing that the partial derivatives are indeed continuous at any point $x$ (in the topology of $\cS^p\times\cH^p$).
We do this by proving the limit result on the difference of gradients. 

This estimate is easy to obtain as it follows from the techniques used throughout Section 3. We do not give the full details. Let $t\in[0,T]$, $x,x'\in\IR^m$ then the BSDE for the difference $\delta \nabla Y=\nabla_x Y^{t,x}-\nabla_x Y^{t,x'}$ (we define as well $\delta \nabla Z=\nabla_x Z^{t,x}-\nabla_x Z^{t,x'}$) is given by
\begin{align}
\label{eq:aux:pseudocontinuitynablaY}
\delta \nabla Y_s= \delta \eta -\int_s^T \delta \nabla Z_r \udw_r+ \int_s^T [F_s+K^{\cY}_s\delta \nabla Y_r + K^{\cZ}_s\delta \nabla Z_r ]\ud r,
\end{align}
with $\delta \eta =(\nabla_x g)(X^x_T) \nabla_x X^x_T-(\nabla_x g)(X^{x'}_T) \nabla_x X^{x'}_T$, $K^{\cY}= \nabla_y f(\cdot,\Theta^x_\cdot)$, $K^{\cZ}= \nabla_z f(\cdot,\Theta^x_\cdot)$ and $\nabla_x\Theta^{x'}=(\nabla_x X^{x'}, \nabla_x Y^{x'}, \nabla_x Z^{x'})$ we define the process  $F:=\nabla_x f(\cdot,\Theta^x)(\nabla_x X^x-\nabla_x X^{x'})+\langle \nabla_x f(\cdot,\Theta^x)-\nabla_x f(\cdot,\Theta^{x'}) , \nabla_x \Theta^{x'}\rangle$.

The results we have seen so far easily imply (see Theorem \ref{theo:FBSDEexistenceanduniqueness}) that $K^{\cY}$ and $K^{\cZ}$ belong to $\cHBMO$ with norm uniformly bounded in $x$. From this structure it is easy to see that the technique used to prove Corollary \ref{maincoro:ch3:aprioriestimate} can be applied to this situation, yielding for a $\beta>1$ related to the BMO norms of the relevant processes and a constant $C$ independent of $x,x'$ that
\begin{align*}
\|\delta \nabla Y\|_{\cS^p}^p
&\leq 
C\IE\Big[\,|\delta \eta|^{p\beta} + \Big(\int_0^T |F_s|\uds \Big)^{p\beta}\Big]^{\frac1{\beta}}.
\end{align*}
We point out that although \eqref{eq:aux:pseudocontinuitynablaY} is a multidimensional equation and our trick consists of a measure change, the measure change is uniform across all entries of $\nabla Y$, namely for each entry it is always the process $K^{\cZ}*W$ that is at the origin of the measure change.

Given the already proved (see Corollary \ref{coro:conituityofYinx}) continuity of $X^x$, $Y^x$ and $Z^x$ in $x$, the continuity of the derivatives of $f$ (by assumption), one can easily conclude via the dominated convergence theorem the sought result. This type of arguments have been used in Section 3.1 of \cite{dosreis2011} or \cite{ImkellerDosReis2010}.
\end{proof}

\section{Proof of Theorem \ref{theo:proofofviscosity}}
\begin{proof}[Proof of Theorem \ref{theo:proofofviscosity}]
Let $(t,x)\in[0,T]\times \IR^m$. We have seen in previous results that the process $Y^{t,x}$ has continuous path as well as a continuous dependency in $x$ (see Corollary \ref{coro:conituityofYinx}). Uniform boundedness of $u$ follows from that of $Y$ implied by Proposition \ref{theo:FBSDEexistenceanduniqueness}.

To prove that $u$ is a viscosity solution we need to prove that it is a viscosity sub- and supersolution. We prove only that it is a subsolution since the supersolution part of the proof if very similar.

To prove that $u$ is a viscosity subsolution, let $\varphi \in C^{1,2}([0,T]\times \IR^m,\IR^m)$ and take $(\bar{t},\bar{x})\in[0,T]\times \IR^m$ such that $(\bar{t},\bar{x})$ is a local maximum of $(u-\varphi)$. We assume that without loss of generality that they touch, i.e. that $u(\bar{t},\bar{x}) = \varphi( \bar{t}, \bar{x} )$. Suppose that\footnote{Compare with the definition of viscosity subsolution: Take $u\in C([0,T]\times\IR^m, \IR)$, $\varphi\in C^{1,2}([0,T]\times \IR^d)$ and $(t,x)\in[0,T]\times \IR^m$. Then $u$ is a viscosity \emph{subsolution} of \eqref{eq:thePDE} if for all $x$ it holds that $u(T,x)\leq g(x)$ and if for an $\varphi$ and a pair $(t,x)$ which is a local maximum of $u-\phi$, it holds that $-\partial_t \varphi(t,x)-\cL \varphi (t,x) - f\big(t, x, u(t,x), (\nabla \varphi \sigma)(t,x)\big) \leq 0$.
} 
\[
\partial_t \varphi(\bar{t},\bar{x})+\cL \varphi (\bar{t},\bar{x}) + f\big(\bar{t},\bar{x}, u(\bar{t},\bar{x}), (\nabla \varphi \sigma)(\bar{t},\bar{x})\big) < 0,
\]
and we now argue by contradiction.

Since $[0,T]\times \IR^m$ is a separable space we can define a neighborhood of $(\bar{t},\bar{x})$ in the following way: let $0< \alpha \leq T-\bar{t}$ be such that for all $\bar{t}\leq s\leq \bar{t}+\alpha$, $|y-\bar{x}|\leq \alpha$, $u(s,y)\leq \varphi(s,y)$ and 
\begin{align}
\label{eq:viscosolutionstrict}
\partial_t \varphi (s,y)+\cL \varphi (s,y) + f\big(s,y, u(s,y), (\nabla \varphi \sigma)(s,y)\big) &< 0,
\end{align}
and define
\[
\tau = \inf\{s\geq \bar{t}:\, |X^{\bar t, \bar x}_s-\bar x|\geq \alpha\} \land (\bar t + \alpha).
\]
The pair
\[
(\widetilde Y_s, \widetilde Z_s):=\big (Y^{\bar t, \bar x}_{s\land \tau}, \indicfunc_{[\bar t,\tau]}(s) Z^{\bar t, \bar x}_{s\land \tau} \big),\quad s\in[\bar t,\bar t+\alpha]
\]
clearly solves the BSDE on $[\bar t, \bar t + \alpha]$
\[
\widetilde Y_s = u(\tau, X_\tau^{\bar t,\bar x})
+\int_s^{\bar t+\alpha} \indicfunc_{[\bar t,\tau]}(r)f\big(r,X_r^{\bar t,\bar x}, u(r,X_r^{\bar t,\bar x}) , \widetilde Z_r\big)\ud r-\int_s^{\bar t + \alpha} \widetilde{Z}_r\udw_r.
\]
On the other hand, since $\varphi\in C^{1,2}$ we have by It\^o's formula that 
\[
(\widehat Y_s, \widehat Z_s):=\big (\varphi(s, X^{\bar t, \bar x}_{s\land \tau}, \indicfunc_{[\bar t,\tau]}(s) (\nabla \varphi \sigma)(s,X^{\bar t, \bar x}_{s\land \tau}) \big),\quad s\in[\bar t,\bar t+\alpha]
\]
solves the BSDE for $s\in[\bar t,\bar t+\alpha]$
\[
\widehat Y_s = \varphi(\tau, X_\tau^{\bar t,\bar x})
-\int_s^{\bar t+\alpha} \indicfunc_{[\bar t,\tau]}(r)\big( \partial_t  \varphi +\cL \varphi\big)(r,X_r^{\bar t,\bar x})\ud r-\int_s^{\bar t + \alpha} \widehat{Z}_r\udw_r.
\]
At this point we need only to use a comparison result that allows us to conclude that $\widetilde{Y}<\widehat{Y}$ and hence that $u(\bar t, \bar x)< \varphi(\bar t, \bar x)$ to find the contradiction with our assumption that $u(\bar t,\bar x)= \varphi(\bar t, \bar x)$. Unfortunately we cannot apply directly our comparison Theorem \ref{maintheo:ch3:comparison}, because we do not know if the BSDE for $(\widehat Y,\widehat Z)$ satisfies its assumptions.

We start by remarking that since $\varphi \in C^{1,2}$ and within the neighborhood around $(\bar t,\bar x)$ defined above,  $\varphi(r,X_r^{\bar t,\bar x})$ and $(\partial_t \varphi + \cL \varphi)(r,X_r^{\bar t,\bar x})$ are continuous functions on a compact set and hence bounded! This allows us to trivially conclude (e.g.~Theorem \ref{theo:2.3Koby2000}) that $(\widehat Y, \widehat Z)\in\cS^\infty\times \cHBMO$ on the time interval $[\bar t, \bar t+\alpha]$. It is also clear that $(\widetilde Y, \widetilde Z)$ belongs to $\cS^\infty\times \cHBMO$ simply because it is the solution of \eqref{FBSDE:SDE}, \eqref{FBSDE:BSDE}. The BSDE for the difference $\widetilde Y-\widehat Y$ is
\begin{align*}
\widetilde Y_s - \widehat Y_s 
&=u(\tau, X_\tau^{\bar t,\bar x})- \varphi(\tau, X_\tau^{\bar t,\bar x})
-\int_s^{\bar t + \alpha} \widetilde{Z}_r-\widehat{Z}_r\udw_r
\\
& 
+\int_s^{\bar t+\alpha} \indicfunc_{[\bar t,\tau]}(r)\big[
f\big(r,X_r^{\bar t,\bar x}, u(r,X_r^{\bar t,\bar x}) , \widetilde Z_r\big) + \big( \partial_t  \varphi +\cL \varphi\big)(r,X_r^{\bar t,\bar x})\big]\ud r.
\end{align*}
Adding and subtracting a term $f\big(r,X_r^{\bar t,\bar x}, u(r,X_r^{\bar t,\bar x}) , \widehat Z_r\big)$ to the driver of the BSDE for the difference $\widetilde Y - \widehat Y$ we can apply arguments similar to those of the proof of Theorem \ref{maintheo:ch3:comparison} and conclude from the fact that $u\leq \varphi$ and the strictness of \eqref{eq:viscosolutionstrict} that $\widetilde Y < \widehat Y$, i.e. $u(\bar t,\bar x)<\varphi(\bar t,\bar x)$. This contradicts the assumption and hence we can conclude that $u(t,x)=Y^{t,x}_t$ is indeed a viscosity subsolution of the PDE. 

The arguments to show that $u$ is a supersolution are similar. A combination of the two resutls yields that $u$ is a viscosity solution of \eqref{eq:thePDE}.1
\end{proof}

\bibliographystyle{alpha}



\end{document}